\documentclass{imsart}
\RequirePackage{natbib}
\usepackage{a4wide}
\usepackage{amsfonts}
\usepackage{amsmath}
\usepackage{amsthm}
\usepackage{bbm}
\usepackage{bm}
\usepackage{geometry}
\usepackage{hyperref}
\usepackage{subcaption}
\usepackage{tikz}
\usepackage{xcolor}

\newcommand{\E}{\mathbb{E}}
\newcommand{\N}{\mathbb{N}}
\renewcommand{\P}{\mathbb{P}}
\newcommand{\R}{\mathbb{R}}
\newcommand{\sign}{\operatorname{sign}}

\newcommand{\1}{\mathbbm{1}}
\newcommand{\diag}{\operatorname{diag}}
\newcommand{\half}{\mbox{$\frac 1 2 $}}
\newcommand{\quarter}{\mbox{$\frac 1 4$}}

\newtheorem{corollary}{Corollary}
\newtheorem{lemma}{Lemma}
\newtheorem{proposition}{Proposition}
\newtheorem{observation}{Observation}
\newtheorem{assumption}{Assumption}
\newtheorem{theorem}{Theorem}

\theoremstyle{definition}

\theoremstyle{remark}

\newtheorem{remark}{Remark}

\begin{document}
 
\begin{frontmatter}
 
\title{A piecewise deterministic scaling limit of Lifted Metropolis-Hastings in the Curie-Weiss model}
\runtitle{Piecewise deterministic scaling limit of Lifted Metropolis-Hastings}

\begin{aug}
  \author{\fnms{Joris}  \snm{Bierkens}\corref{}\thanksref{t2} \ead[label=e1]{j.bierkens@warwick.ac.uk}}
  \and	
  \author{\fnms{Gareth} \snm{Roberts}\thanksref{t2} \ead[label=e2]{gareth.o.roberts@warwick.ac.uk}}
  
  \thankstext{t2}{Both authors would like to thank the EPSRC for support under grants EP/D002060/1 (CRiSM) and EP/K014463/1 (iLike)}
  
  \runauthor{J. Bierkens and G. Roberts}
  
  \affiliation{University of Warwick, Department of Statistics}
  
  \address{Coventry, CV4 7AL, United Kingdom \\ \printead{e1,e2}}
\end{aug}

\begin{abstract}
In \cite{TuritsynChertkovVucelja2011} a non-reversible Markov Chain Monte Carlo (MCMC) method on an augmented state space was introduced, here referred to as Lifted Metropolis-Hastings (LMH). A scaling limit of the magnetization process in the Curie-Weiss model is derived for LMH, as well as for Metropolis-Hastings (MH). The required jump rate in the high (supercritical) temperature regime equals $n^{1/2}$ for LMH, which should be compared to $n$ for MH. At the critical temperature the required jump rate equals $n^{3/4}$ for LMH and $n^{3/2}$ for MH, in agreement with experimental results of \cite{TuritsynChertkovVucelja2011}.
The scaling limit of LMH turns out to be a non-reversible piecewise deterministic exponentially ergodic `zig-zag' Markov process.
\end{abstract}




\end{frontmatter}

\section{Introduction}

Markov chain Monte Carlo (MCMC, \cite{Metropolis1953}) has been extremely successful in providing a generic simulation framework with wide-ranging applications. It works by composing collections of move types, each which leave the target distribution of interest invariant. Invariance is assured through detailed balance making the building blocks of MCMC reversible, giving advantages in terms of accessibility to mathematical investigation and practicality of implementation. Yet there is a growing interest in the phenomenon that, where comparative mathematical studies are possible, non-reversible Markov chains often outperform their reversible counterparts. 

A fundamental approach for obtaining non-reversible Markov processes is by `lifting' or `augmenting' the state space. In this case the states are augmented by one or more additional variables, which can often be interpreted as e.g. momentum or direction.
Let us provide a (non-exhaustive) overview of the literature concerning this approach.
In \cite{Chen1999}  it is shown that `lifting' a Markov chain may result in a reduced mixing time that is at best of order square root of the original mixing time. In order to achieve this improvement, a non-reversible lifting is required.
In \cite{DiaconisHolmesNeal2000} a simple reversible Markov chain on a finite state space of size $n$, is lifted to a non-reversible Markov chain on a space of size $2n$. It is shown that this construction reduces the mixing time of the chain from $O(n^2)$ to $O(n)$.  In \cite{TuritsynChertkovVucelja2011} a non-reversible lifting of Metropolis-Hastings is introduced, which we will refer to as \emph{Lifted Metropolis-Hastings (LMH)}, and applied to the Ising model on a fully connected graph (i.e. the \emph{Curie-Weiss model}). In a numerical experiment it appears that at the critical temperature, the `relaxation time' is reduced from $O(n^{1.43})$ to $O(n^{0.85})$, where $n$ denotes the number of spins. 
The `lifting approach' is not the only way of obtaining non-reversible Markov processes. For interesting approaches to constructing and analyzing the benefits of non-reversible Markov processes, see e.g. \cite{Hwang1993}, \cite{Sun2010}, \cite{Chen2013}, \cite{ReyBelletSpiliopoulos2015}, \cite{Bierkens2015}, \cite{Lelievre2013}, \cite{DuncanLelievrePavliotis2015}.

It is the goal of this paper to shed light on the general theory of lifted non-reversible Markov chains, and in particular on the recent experimental result of \cite{TuritsynChertkovVucelja2011} mentioned above. This is achieved by obtaining a scaling limit of Lifted Metropolis-Hastings, in its application to the Curie-Weiss model. This scaling limit may be compared to a similar scaling limit for (classical) Metropolis-Hastings.

Initiated by \cite{RobertsGelmanGilks1997}, a large amount of understanding of particular Markov Chain Monte Carlo (MCMC) algorithms has been obtained by identifying a suitable diffusion limit: Given a sequence of Markov chains of increasing size or dimensionality $n$, a suitable scaling of the state space and of the amount of steps per unit time interval is determined. As $n$ tends towards infinity, the scaled Markov process converges (in the sense of weak convergence on Skorohod path space) to a diffusion process, which is often of an elementary nature. In particular the required number of Markov chain transitions per unit time interval as a function of $n$ provides a fundamental measure of the speed of the Markov chain.

The Curie-Weiss model is an exchangeable probability distribution on $\{-1, 1\}^n$ which depends on two parameters, the `external field' $h$, and the `inverse temperature' $\beta$ (which describes interactions between components). At inverse temperature $\beta = 1$ the model undergoes a phase transition. This results in differences in behaviour for $\beta < 1$, $\beta = 1$ and $\beta > 1$, and we shall analyse the behaviour of both standard Metropolis-Hastings and Lifted Metropolis-Hastings in the first two of these cases.  
We will determine a scaling limit of Metropolis-Hastings (\cite{Metropolis1953, Hastings1970}) as well as Lifted Metropolis-Hastings (\cite{TuritsynChertkovVucelja2011}) for the magnetization in the Curie-Weiss model, for the supercritical temperature regime ($\beta < 1$) with external field $h \in \R$ and at the critical temperature ($\beta = 1$), without external field ($h = 0$). To obtain these results we depend on non-asymptotic concentration results of \cite{Chatterjee2007},  \cite{ChatterjeeDey2010}. The case of subcritical temperature ($\beta > 1$) is typically more difficult to analyse. In this paper we do not obtain results for this case because, as far as we know, no non-asymptotic concentration results are available.

In both the supercritical and critical cases our results demonstrate that the lifted chain convergence is an order of magnitude faster (as a function of dimension $n$) than the regular Metropolis-Hastings algorithm as is summarised below.

\begin{table}[ht]
\begin{tabular}{ | l || c | c | } \hline  & $\beta <1$ & $\beta =1$ \\ \hline \hline Metropolis-Hastings & $O(n)$ (Theorem~\ref{thm:diffusion-limit-mh-high-temperature}) & $O(n^{3/2})$ (Theorem~\ref{thm:diffusion-limit-mh-critical-temperature})\\
 Lifted Metropolis-Hastings & $O(n^{1/2})$ (Theorem~\ref{thm:LMH-supercritical-temperature}) &  $O(n^{3/4})$ (Theorem~\ref{thm:LMH-critical-temperature}) \\ \hline\end{tabular}
 \label{tab:results}
\end{table}

The results for Lifted Metropolis-Hastings are surprising since it would require at least $O(n)$ iterations to update each component. Therefore magnetization is converging significantly more rapidly than should be expected. This is explained by the strong concentration of the magnetization around its  mean, so that only relatively few spin updates suffice to update the magnetization at the appropriate scale.

As is common for weak limits of light-tailed Metropolis-Hasting algorithms, the limits of Metropolis-Hastings for Curie-Weiss are simple univariate diffusion processes. Interestingly, in determining the scaling limit of Lifted Metropolis-Hastings we obtain an elementary Markov process which has so far received only very limited attention in the literature. The limit process is a one-dimensional piecewise deterministic Markov process which we will refer to as a \emph{zig-zag process}: the process moves at a deterministic and constant speed, until it switches direction and moves at the same speed but in the opposite direction. The switching occurs at a time-inhomogeneous rate which is directly related to the derivative of the density function of its stationary distribution. We analyse this zig-zag process in some detail, establishing in particular exponential ergodicity under mild conditions.

Piecewise deterministic Markov processes were first introduced in \cite{Davis1984} and discussed extensively in  \cite{Davis1993}. A zig-zag process with a constant switching rate appears in \cite{Goldstein1951} and is discussed further in \cite{Kac1974}. 
A similar process on the torus is discussed in \cite{MicloMonmarche2013}. In \cite{PetersDeWith2012} a multi-dimensional version of the zig-zag process with space inhomogeneous switching rates is introduced and used for MCMC. This MCMC method is analysed in detail in \cite{Bouchard-Cote2015}. \cite{Monmarche2014} discusses the use of the one-dimensional zig-zag process for simulated annealing, and in \cite{Monmarche2014b} the exponential ergodicity of the zig-zag process is discussed in case of bounded switching rates. In \cite{Fontbona2012} and \cite{Fontbona2015} the exponential ergodicity of the one-dimensional zig-zag process is discussed under seemingly stronger conditions than in the current paper.

\subsection{Outline}
This article is structured as follows. In Section~\ref{sec:preliminaries} we briefly provide the necessary background on Metropolis-Hastings (MH), Lifted Metropolis-Hastings (LMH, based on \cite{TuritsynChertkovVucelja2011}), and the Curie-Weiss model, including the relatively recent non-asymptotic concentration results of \cite{Chatterjee2007}, \cite{ChatterjeeDey2010}. Also we briefly describe the basic random walk Markov chain, used as proposal chain in MH and LMH, in terms of magnetization.

In Section~\ref{sec:diffusion-limit-mh} we consider the time evolution of the magnetization as $n \rightarrow \infty$ for MH applied to the Curie-Weiss model. By a suitable rescaling of both space (i.e. the magnetization variable) and time (i.e. the jump rate within a unit time interval) we obtain a diffusion limit of this stochastic process, at supercritical temperature $\beta < 1$ (Theorem~\ref{thm:diffusion-limit-mh-high-temperature}) as well as at critical temperature, $\beta = 1$, $h = 0$ (Theorem~\ref{thm:diffusion-limit-mh-critical-temperature}). It is perhaps not very surprising that this diffusion limit corresponds to the Langevin diffusion of the known limiting distribution, i.e. a Gaussian distribution in case $\beta < 1$ and non-Gaussian in case $\beta = 1$. Also not surprisingly, the required jump rate to obtain this diffusion limit is in line with known results on mixing time for Curie-Weiss (\cite{Levin2009b}, \cite{Ding2009}): if $\beta < 1$, the required speed up is equal to a factor $n$, while for $\beta = 1$ and $h = 0$ the required speed up is equal to a factor $n^{3/2}$.

The main result of this paper may be found in Section~\ref{sec:LMH-CW-scaling-limit}. In this section  we obtain the scaling limit of the magnetization for LMH applied to Curie-Weiss, again for the cases $\beta < 1$ (Theorem~\ref{thm:LMH-supercritical-temperature}) and $\beta = 1$, $h = 0$ (Theorem~\ref{thm:LMH-critical-temperature}).
The limiting process is a piecewise deterministic Markov process which has received only a small amount of attention in the mathematics and physics literature. Naturally it has the same limiting invariant distribution as for Metropolis-Hastings. Interestingly, the required time scaling for LMH corresponds exactly to the square root of the time scaling for MH: this time scaling is $n^{1/2}$ for $\beta < 1$ and $n^{3/4}$ for $\beta = 1$, $h = 0$. This `square root' improvement is in agreement with the theory developed in \cite{Chen1999} and in line with the numerical result of \cite{TuritsynChertkovVucelja2011}.

In Section~\ref{sec:limiting-process} the limiting `zig-zag' process is analysed in detail. First the process is generalized to allow for general one-dimensional invariant distributions satisfying mild conditions on the derivative of the density function. In particular it is established that this process is a non-explosive process satisfying the strong Markov property (Proposition~\ref{prop:construction}) which is weak Feller (Proposition~\ref{prop:feller}) but not strong Feller (Observation~\ref{obs:not-strong-Feller}). A straightforward relation between the switching  rate of the process and its invariant distribution is obtained in Proposition~\ref{prop:invariant-measure}. Furthermore under a reasonable strengthening of the assumptions exponential ergodicity is obtained (Theorem~\ref{thm:exponential-ergodicity}).

Section~\ref{sec:proofs} is devoted to the proofs of the mentioned results, including necessary technical lemmas. In particular let us mention the following intermediate results: The Feller property is obtained by a coupling argument (Proof of Proposition~\ref{prop:feller}), it is shown that all compact sets are `petite sets' for the zig-zag process (Lemma~\ref{lem:compact-sets-are-petite}), and a Foster-Lyapunov function is constructed to establish exponential ergodicity (Lemma~\ref{lem:lyapunov}).

\section{Preliminaries}
\label{sec:preliminaries}

\subsection{Metropolis-Hastings (MH)}
For a given proposal transition probability matrix $Q$ and target distribution $\pi$ on a discrete state space $S$, the Metropolis-Hastings transition probabilities are given for $x \neq y$ by
\begin{equation} \label{eq:MH} P(x,y) = \left\{ \begin{array}{ll} Q(x,y) \left ( 1 \wedge \frac{\pi(y) Q(y,x)}{\pi(x) Q(x,y)} \right) \quad & \mbox{if $Q(x,y) > 0$,} \\
                       0 \quad & \mbox{otherwise}.
                      \end{array} \right. \end{equation} 
As is well established, the transition probabilities $P$ are reversible with respect to $\pi$, i.e. $\pi(x) P(x,y) = \pi(y) P(y,x)$ for all $x, y \in S$. This implies that $\pi$ is invariant for $P$.

\subsection{Lifted Metropolis-Hastings (LMH)}
\label{sec:LMH}
In \cite{TuritsynChertkovVucelja2011} a non-reversible chain $T$ is constructed with invariant distribution $\half (\pi, \pi)$ on an augmented state space $S^{\sharp} := S \times \{-1, +1 \}$. The set $S \times \{+1 \}$ is called the \emph{forward replica} and $S \times \{-1 \}$ is called the \emph{backward replica}.
The construction is as follows. Let $T^{+}(x,y)$ and $T^{-}(x,y)$, $x \neq y$, denote probabilities satisfying the following conditions:
\begin{itemize}
 \item[(i)] $T^{\pm}(x,y) \geq 0$ for all $x,y \in S$, $x \neq y$;
 \item[(ii)] $\sum_{y \in S, y \neq x} T^{\pm}(x,y) \leq 1$ for all $x \in S$;
 \item[(iii)] skew detailed balance: 
 \begin{equation} \label{eq:skew-detailed-balance} \pi(x) T^{+}(x,y) = \pi(y) T^{-}(y,x) \quad \mbox{for all $x \neq y$.}\end{equation}
\end{itemize}

The values $T^{+}$ and $T^{-}$ will represent transition probabilities within the respective replicas. 
Define transition probabilities between the forward and backward replicas by
\begin{equation} \label{eq:transitions-between-replicas}
\begin{aligned}  T^{-+}(x) & = \max \left(0, \sum_{\substack{y \in S \\ y \neq x}} \left( T^{+}(x,y) - T^{-}(x,y) \right) \right), \\
T^{+-}(x) & = \max \left(0, \sum_{\substack{y \in S \\ y \neq x}} \left(T^{-}(x,y) - T^{+}(x,y)\right) \right). 
\end{aligned}
\end{equation}
Finally for $x \in S$, define $T^{+}(x,x)$ and $T^{-}(x,x)$ by 
\[ T^{+}(x,x) = 1 - T^{+-}(x) - \sum_{\substack{y \in S \\ y \neq x}} T^{+}(x,y) \quad \mbox{and} \quad T^{-}(x,x) = 1  - T^{-+}(x) - \sum_{\substack{y \in S \\ y \neq x}} T^{-}(x,y),\]
so that the rows sums equal 1.
Define the full matrix of transition probabilities $T$ with state space $S \times \{ -1, +1 \}$ by
\begin{align*} T((x,-1), (y,-1)) & = T^{-}(x,y), \quad & T((x,+1), (y,+1)) & = T^{+}(x,y), \\
T((x,+1),(x,-1)) & = T^{+-}(x), \quad & T((x,-1),(x,+1)) & =T^{-+}(x), \\
T((x,-1),(y,+1)) & = 0 \quad \mbox{for $x \neq y$}, \quad & T((x,+1),(y,-1)) & = 0 \quad \mbox{for $x \neq y$},
\end{align*}
or in matrix notation,
\[ T = \begin{pmatrix} T^{+} & \diag(T^{+-}) \\
                 \diag(T^{-+}) & T^{-}  \end{pmatrix}.\]

A few important properties of $T$ are stated in the following proposition. Most importantly, the marginal invariant distribution of $T$ over $S$ is equal to $\pi$.

\begin{proposition}
\label{prop:turitsyn}
Let $T$ be as constructed above. Then
\begin{itemize}
 \item[(i)] $T$ is a Markov transition matrix,
 \item[(ii)] $T$ has invariant probability distribution on $S \times \{-1,+1 \}$ equal to $\half (\pi,\pi)$, and
 \item[(iii)] if, for some $x,y \in S$, $T^{+}(x,y) \neq T^{-}(x,y)$, then $T$ is not reversible with respect to its invariant distribution.
\end{itemize}
\end{proposition}

\begin{proof}
The proofs of these results are immediate.
\end{proof}

\begin{remark} \label{rem:TCV-construction}
Once $T^{+}$ and $T^{-}$ are picked, $T$ is fixed according to the definitions above. However, there is still freedom in choosing $T^+$ and $T^-$ satisfying~\eqref{eq:skew-detailed-balance}. In \cite{TuritsynChertkovVucelja2011} and here, $T$ is fixed as follows. Suppose that $P$ is a transition matrix on $S$ that is reversible with respect to $\pi$, and let $\eta : S \rightarrow \R$. Now define the off-diagonal components of $T^{\pm}$ by
\begin{align*} T^{+}(x,y) & = \left\{ \begin{array}{ll} P(x,y) \quad & \mbox{if $\eta(y) \geq \eta(x)$,} \\
 0 \quad & \mbox{if $\eta(y) < \eta(x)$}, \end{array} \right. \quad \mbox{and} \\
 T^{-}(x,y) & = \left\{ \begin{array}{ll}  0 \quad & \mbox{if $\eta(y) > \eta(x)$,} \\ P(x,y) \quad & \mbox{if $\eta(y) \leq \eta(x)$.} \\
\end{array} \right.
\end{align*}
Then $T^{\pm}$ satisfies the skew detailed balance condition~\eqref{eq:skew-detailed-balance}. This way, Lifted Metropolis-Hastings creates a non-reversible lifted chain $T$ out of a given reversible chain $P$, which has (marginally) the same invariant distribution as $P$.
In particular, this construction may be applied to the Metroplis-Hastings transition probabilities $P$ given by~\eqref{eq:MH}.
\end{remark}

\begin{remark}
\label{rem:freedom-in-replica-switches}
There is some freedom in the choice of transition probabilities between replicas, i.e. $T^{+-}$, $T^{-+}$. In general, transition probabilities between replicas need to satisfy the conditions
\begin{equation} \label{eq:condition-replica-transitions}T^{+-}(x) - T^{-+}(x) = \sum_{\substack{y \in S \\ y \neq x}} \left[ T^{-}(x,y) - T^{+}(x,y) \right],\end{equation}
in order for $\half(\pi, \pi)$ to be invariant.
Here, as in \cite{TuritsynChertkovVucelja2011}, we choose~\eqref{eq:transitions-between-replicas}. See \cite{SakaiHukushima2013} for other variants.
\end{remark}

\subsection{The Curie-Weiss model}

Let $S^n := \{-1, 1\}^n$ and let target invariant distributions $\pi^n$ on $S^n$ be given by \begin{equation}
\label{eq:inv-distribution-hamiltonian} \pi^n(x) = Z_n \exp(-\beta H^n(x)),\end{equation}
with 
\begin{equation} \label{eq:hamiltonian} H^n(x) = -\frac 1 {2n} \sum_{i,j=1}^n  x_i x_j - h \sum_{i=1}^n x_i,\end{equation}
where $(Z_n)$ are normalization constants, $\beta$ is a parameter usually referred to as inverse temperature, and $h \in \R$ a parameter known as the external magnetization. 
As remarked in the introduction, we will later specialize to the case $0 \leq \beta \leq 1$, but for now we allow general $\beta \geq 0$. 
Define the \emph{magnetization} $m^n : \{-1,1\}^n \rightarrow \R$ by $m^n(x) = \frac 1 n \sum_{i=1}^n x_i$. The crucial observation for the Curie-Weiss model is that the Hamiltonian may be expressed in terms of $m$, as
\begin{equation} \label{eq:hamiltonian-magnetization} H^n(x) = - n \left( \half (m^n(x))^2 + h m^n(x) \right).\end{equation}
We may consider $m^n$ and other mappings from $S^n$ into $\R$ as random variables on the probability space $(S^n, \pi^n)$; in particular we will suppress the dependence on $x \in S^n$ where this does not cause confusion.

For $0 \leq \beta \leq 1$, as well as for $\beta > 1$ and $h \neq 0$, there exists a unique $m_0 = m_0(h, \beta)$ around which the magnetization will concentrate. The value of $m_0$ can be obtained as the unique minimizer of
\begin{equation} \label{eq:curie-weiss-functional} i(m) = -\left(\half \beta m^2 + \beta h m \right) + \frac{1-m}{2} \log(1- m) + \frac{1 + m}{2} \log(1 +m), \quad m \in (-1,1).\end{equation}
This value $m_0$ satisfies 
\begin{equation} \label{eq:curie-weiss}\beta m_0 + \beta h = \half \log \frac{1+ m_0}{1 - m_0},\end{equation}
or equivalently $m_0 = \tanh \left( \beta (m_0 + h) \right)$. In case $\beta > 1$ and for $h$ sufficiently small there exist two other solutions to~\eqref{eq:curie-weiss} but these are not global minima of~\eqref{eq:curie-weiss-functional}. As $h \rightarrow 0$, $m_0(h,\beta) \rightarrow 0$. For $h = 0$ or $\beta = 0$, $m_0 = 0$. For $h \neq 0$, the sign of $m_0$ is equal to the sign of $h$. These results are well known, see e.g. \cite[Section IV.4]{Ellis2006}. 

As $n \rightarrow \infty$, the random variables $m^n$ will be increasingly concentrated around $m_0$:
\begin{proposition}[Concentration for Curie-Weiss]
\label{prop:concentration-Curie-Weiss}
\begin{itemize}
\item[(i)] For all $\beta \geq 0$, $h \in \R$ and $t \geq 0$,
\[ \pi^n \left(|m^n - \tanh(\beta (m^n + h))| \geq \frac{\beta}{n} + \frac{t}{\sqrt{n}} \right) \leq 2 \exp\left( -\frac{t^2}{4(1+\beta)} \right).\]
\item[(ii)] If $h = 0$ and $\beta = 1$, then there exists a constant $c > 0$ such that for any $n \in \N$  and $t \geq 0$,
\[ \pi^n(|m^n| \geq t^{1/4}) \leq 2 e^{-c n t}.\]
\end{itemize}
\end{proposition}
\begin{proof}
Claim (i) is \cite[Proposition 1.3]{Chatterjee2007}. Claim (ii) is a simple consequence of \cite[Proposition 5]{ChatterjeeDey2010}.
\end{proof}
 
\begin{remark}
In case $\beta > 1$ and $h \neq 0$ there is a unique global minimum of~\eqref{eq:curie-weiss-functional}. However, to develop scaling limits for Metropolis-Hastings and Lifted Metropolis-Hastings we require non-asymptotic concentration results as given in Lemmas~\ref{lem:concentration-high-temperature} and~\ref{lem:concentration-critical-temperature}, which are based upon Proposition~\ref{prop:concentration-Curie-Weiss}. Even though Proposition~\ref{prop:concentration-Curie-Weiss} includes the case $\beta > 1$, the proof of Lemma~\ref{lem:concentration-high-temperature} seems to depend crucially on the assumption that $\beta < 1$. Therefore we have to restrict our attention to $0 \leq \beta < 1$ (along with the critical case $h =0, \beta = 1$).
\end{remark}

As quantity of interest (which is a necessary ingredient in the formulation of the lifted Markov chain, see Remark~\ref{rem:TCV-construction}), we will consider suitably shifted and renormalized magnetization,
\[ \eta^n(x) := n^{\gamma} (m^n(x) - m_0), \quad x \in S^n.\]
In view of Proposition~\ref{prop:concentration-Curie-Weiss}, for $\eta^n$ to be of $O(1)$ as $n \rightarrow \infty$, we will need to choose $\gamma = 1/2$ for $0 \leq \beta < 1$ and $\gamma = 1/4$ for $\beta = 1$. For smaller choices of $\gamma$ any limiting random variable would be trivially concentrated at a single point, whereas for larger choices of $\gamma$ a suitable limiting random variable would not exist. The precise concentration statements we will use are given in Lemmas~\ref{lem:concentration-high-temperature} (for $0 \leq \beta < 1$) and~\ref{lem:concentration-critical-temperature} (for $\beta = 1$).
For now we will only assume that $\gamma \in (0, 1)$.

Rather than using $x$ as state space variable, it will be useful to express all quantities and probabilities in terms of $\eta^n(x)$.  For example, the Hamiltonian $H^n$ can be re-expressed in terms of $\eta^n(x)$ by $H^n(x) = c_n + \Phi^n(\eta^n(x))$, where the constants $c_n$ do not depend on $\eta^n$, and
\begin{equation}
 \label{eq:Phi}
 \Phi^n(\eta) := - \half n^{1-2 \gamma} \eta^2 - n^{1-\gamma}(m_0 + h) \eta, \quad \eta \in \R.
\end{equation}

\subsection{Random walk on the discrete hypercube}
\label{sec:randomwalk}
Consider the Markov transition probabilities on $S^n = \{-1,1\}^n$ given by
\[ \mathrm{Prob}(x \rightarrow y) = \left \{ \begin{array}{ll} \frac 1 n \quad & \mbox{when $y = F_k(x)$ for some $k = 1, \dots, n$,} \\
 0 \quad & \mbox{otherwise}. \end{array} \right.
\]
Here $F_k : S^n \rightarrow S^n$ denotes the operation of flipping the sign of $x(k)$, i.e. 
\[ [F_k(x)]_i := \left\{ \begin{array}{ll} x_i \quad & \mbox{for $i \neq k$,}\\
              - x_i \quad & \mbox{for $i = k$.}
                       \end{array} \right. \]
In words, a transition consists of flipping the sign of $x_i$, where $i$ is selected uniformly among $\{ 1, \dots, n \}$. This Markov chain corresponds to a random walk on the discrete hypercube $S^n$.
We will express the above transition probabilities in terms of $\eta = \eta^n(x)$ rather than $x$. For $\eta = \eta^n(x)$, a fraction $\half(1 - m^n(x)) = \half(1 - m_0 - n^{-\gamma} \eta)$ of entries of $x$ has value $-1$, and similarly a fraction $\half(1 + m_0 + n^{-\gamma} \eta)$ has value $+1$. If one entry of $x$ flips, there is a change in $\eta^n$ by $2 n^{\gamma - 1}$. Therefore, for $\eta \in X^n := \eta^n(S^n)$, we define
\begin{equation}
 \label{eq:transition-probs-eta-rw}
 Q^n(\eta, \eta \pm 2n^{\gamma - 1}) := \half ( 1 \mp (m_0 + n^{-\gamma} \eta)),
\end{equation}
and $Q^n(\eta, \zeta) := 0$ for all $\eta, \zeta \in X^n$ for which $|\zeta - \eta| \neq 2n^{\gamma - 1}$. Defined this way, $Q^n$ is a matrix of transition probabilities on $X^n$.

\section{Diffusion limit of Metropolis-Hastings applied to Curie-Weiss}
\label{sec:diffusion-limit-mh}

In this section we consider the limit of Metropolis-Hastings for the Curie-Weiss model as $n \rightarrow \infty$ in terms of the scaled magnetization $\eta^n(x) = n^{\gamma} (m^n(x) - m_0)$. 
 In terms of $\eta^n$, the invariant distribution is given by
\begin{equation} \label{eq:inv-distribution-eta}
\mu^n(\eta) := (\pi^n \circ (\eta^n)^{-1})(\eta) \propto \exp(- \beta \Phi^n(\eta)), \quad \eta \in X^n.\end{equation}
Using the random walk transition probabilities $Q^n$ and the target distribution $\mu^n$ for the Curie-Weiss model, we obtain for the MH transition probabilities 
\begin{equation} \label{eq:MH-definition}  P^n(\eta, \eta \pm 2 n^{\gamma - 1}) = Q^n(\eta, \eta \pm 2 n^{\gamma - 1}) \left( 1 \wedge \exp (\beta \{\Phi^n(\eta) - \Phi^n(\eta \pm 2 n^{\gamma - 1})\}) \right),
\end{equation}
for $\eta \in X^n$, with $\Phi^n$ given by~\eqref{eq:Phi}.

Let $Y^n$ denote the stationary continuous time Markov chain that jumps at rate $n^{\alpha}$ according to $P^n$ with stationary distribution $\mu^n  \propto \exp(-\beta \Phi^n(\eta))$. 
Let $D([0, \infty), \R)$ denote the space of cadlag paths in $\R$, equipped with the Skorohod topology. 
We are now in a position to state our first two results concerning the high-dimensional limit of $Y^n$ in the supercritical and critical cases respectively.

\begin{theorem}[Diffusion limit of Metropolis-Hastings in the supercritical temperature regime]
\label{thm:diffusion-limit-mh-high-temperature}
Suppose $0 \leq \beta < 1$ and $h \in \R$. Suppose $Y^n$ jumps at rate $n$, i.e. we let $\alpha = 1$ in the above definition of $Y^n$. Let the spatial scaling in the transition probabilities $P^n$ be determined by $\gamma = \half$. Then $Y^n$ converges weakly in $D([0,\infty),\R)$ to $Y$, where $Y$ is the stationary Ornstein-Uhlenbeck process satisfying the stochastic differential equation
\[ d Y(t) = -2  l(h,\beta) Y(t) \ d t + \sigma(h,\beta) \ d B(t), \quad Y(0) \sim \mu,\]
and with stationary distribution $\mu$, where $(B(t))$ is a standard Brownian motion, $\mu$ is the centred normal distribution with variance 
\begin{equation} \label{eq:supercritical-variance} v(h, \beta) := \frac{1 - m_0^2(h,\beta)}{1 - \beta(1 - m_0^2(h,\beta))}\end{equation}  and with
\[ \sigma(h, \beta) := 2 \sqrt{1 - |m_0(h,\beta)|} \quad \mbox{and} \quad l(h,\beta) := \frac 1 {1 + |m_0(h,\beta)|} - \beta(1-|m_0(h,\beta)|).\]
\end{theorem}

The proof depends on the convergence of the infinitesimal generator of the Markov chain semigroup, as e.g. \cite[Theorem 1.1]{RobertsGelmanGilks1997}, and is provided in Section~\ref{sec:proofs}.

\begin{theorem}[Diffusion limit of Metropolis-Hastings at the critical temperature]
\label{thm:diffusion-limit-mh-critical-temperature}
Suppose $\beta =1$ and $h = 0$. Suppose $Y^n$ jumps at rate $n^{3/2}$, i.e. we let $\alpha = 3/2$ in the definition of $Y^n$. Let the spatial scaling in the transition probabilities $P^n$ be determined by $\gamma = \quarter$. Then $Y^n$ converges weakly in $D([0,\infty),\R)$ to $Y$, where $Y$ is the stationary Langevin process satisfying the stochastic differential equation
\[ d Y(t) = -(2/3) (Y(t))^3 \ d t + 2  \ d B(t), \quad Y(0) \sim \mu\]
with $(B(t))$ a standard Brownian motion, where $\mu$ is the probability distribution on $\R$ with Lebesgue density 
\[ \frac{d \mu}{d y} = \left( \frac{4}{3}\right)^{1/4}\frac{\exp(-y^4/12)}{\Gamma(1/4)}.\]
\end{theorem}
The expression for the limiting non-Gaussian distribution for the Curie-Weiss model is well known, see e.g. \cite[p. 4]{ChatterjeeDey2010}.

\section{Scaling limit for Lifted Metropolis-Hastings applied to Curie-Weiss}
\label{sec:LMH-CW-scaling-limit}

Carrying out the construction of Section~\ref{sec:LMH}, the Lifted Metropolis(-Hastings) scheme with random walk proposal leads to transition probabilities $T^n$ in the space $X^n \times \{-1, +1\}$ given by 
\begin{equation} \begin{aligned}
\label{eq:LMH-CW-transitions}
 T^n((\eta, +1),(\eta +  2 n^{\gamma -1}, +1)) = p^n_+(\eta), \\
 T^n((\eta, -1), (\eta - 2 n^{\gamma -1},-1)) = p^n_-(\eta), \\
 T^n((\eta, +1),(\eta, -1)) = \max\left(0,p^n_-(\eta) - p^n_+(\eta) \right), \\
 T^n((\eta, -1),(\eta, +1)) = \max\left(0,p^n_+(\eta) - p^n_-(\eta) \right),
\end{aligned} \end{equation}
and all other transition probabilities from $(\eta, \pm 1)$ to a different state are equal to zero.
Here $p^n_{\pm} = P^n(\eta, \eta \pm 2 n^{\gamma - 1})$, with $P^n$ the transition probabilities of MH for the Curie-Weiss model, as given by~\eqref{eq:MH-definition}. Recall that  $p^n_{\pm}$, and hence $T^n$, depends on the choice of the spatial scaling parameter $\gamma$.

Let $(Y^n, J^n)$ denote the stationary continuous time Markov chain which jumps at rate $n^{\alpha}$ according to $T^n$. Let  
\begin{equation}
 \label{eq:LMH-drift}
 a(h,\beta) := 1 - |m_0|
\end{equation}
and let $l(h,\beta)$ be as given in Theorem~\ref{thm:diffusion-limit-mh-high-temperature}. 
In the supercritical temperature regime, with $0 \leq \beta < 1$ and $h \in \R$, the limiting Markov process will be shown to have generator
\begin{equation}
 \label{eq:limiting-generator-LMH} L \varphi(\eta,j) = a(h, \beta) j \frac{\partial \varphi}{\partial \eta} + \max(0, j l(h,\beta) \eta) (\varphi(\eta,-j) - \varphi(\eta, j)),
\end{equation}
with domain 
\[D(L) = \left\{ \varphi : \R \times \{-1,1\} \rightarrow \R, \eta \mapsto \frac{\partial \varphi}{\partial \eta}(\eta, j) \in C_0(\R) \ \mbox{for $j = \pm 1$}\right\},\] 
where  $C_0(\R)$ is the Banach space of continuous functions on $\R$, vanishing at infinity.
This scaling limit is obtained provided we choose the right speed factor: we have to jump at rate $n^{1/2}$. This is formulated in the following theorem.

\begin{theorem}
\label{thm:LMH-supercritical-temperature}
Suppose $0 \leq \beta < 1$ and $h \in \R$. Suppose $(Y^n, J^n)$ jumps at rate $n^{1/2}$, i.e. we let $\alpha = 1/2$ in the definition of $(Y^n, J^n)$. Let the spatial scaling in the transition probabilities $T^n$ be determined by $\gamma = \half$. 
Then $(Y^n,J^n)$ converges weakly in $D([0,\infty),\R \times\{-1,1\})$ to $(Y,J)$, where $(Y,J)$ is the stationary Markov process with generator $L$ and stationary distribution $\half \mu \otimes (\delta_{-1}  + \delta_{+1})$, with $\mu = N(0, v(h,\beta))$ and $v(h,\beta)$ given by~\eqref{eq:supercritical-variance}.
\end{theorem}

It will be established in Section~\ref{sec:limiting-process} that $L$ is the generator of a Markov-Feller process. Let $(Y,J)$ denote the continuous time Markov process with generator $L$. The interpretation of $(Y,J)$ is straightforward: $Y$ moves with constant drift $a(h,\beta)$ in the direction $J$, until it changes its direction to $-J$. The changes in direction occur at events generated by a time inhomogeneous Poisson process with switching rate given by $\max(0, J(t) l(h,\beta) Y(t))$. See Section~\ref{sec:limiting-process} for a detailed discussion of this process.
At the critical temperature we have to jump at a faster rate $n^{3/4}$ to obtain a non-trivial limiting Markov process. The limiting process is slightly different (compared to the supercritical temperature regime) in the sense that it switches replicas at a modified (cubic) rate:

\begin{theorem} 
\label{thm:LMH-critical-temperature}
Suppose $\beta = 1$ and $h = 0$. Suppose $(Y^n, J^n)$ jumps at rate $n^{3/4}$, i.e. we let $\alpha = 3/4$ in the definition of $(Y^n,J^n)$. Let the spatial scaling in the transition probabilities $T^n$ be determined by $\gamma = \quarter$. 
Then $(Y^n,J^n)$ converges weakly in $D([0,\infty),\R \times\{-1,1\})$ to $(Y,J)$, where $(Y,J)$ is the stationary Markov process with generator $L$ given 
\begin{equation}
 \label{eq:limiting-generator-LMH-critical-temperature}
 L\varphi(\eta,j) = j \frac{d \varphi}{d \eta}(\eta,j) + \max(0, 1/3 j \eta^3) (\varphi(\eta,-j)- \varphi(\eta,j)),
\end{equation}
with stationary distribution $\half \mu \otimes (\delta_{-1}  + \delta_{+1})$, where $\mu$ is as in Theorem~\ref{thm:diffusion-limit-mh-critical-temperature}.
\end{theorem}

\begin{remark}
Analogous results can be obtained for the closely related Glauber dynamics and its lifted version. The only difference is that the resulting Langevin diffusion (for Glauber dynamics) and zig-zag process (for lifted Glauber dynamics) are a factor $2/(1+|m_0|) \in (1,2]$ slower than for MH and LMH.
\end{remark}

\section{The limiting zig-zag process}
\label{sec:limiting-process}

In this section we will investigate a generalization of the Markov process with generator~\eqref{eq:limiting-generator-LMH}. Let $E = \R \times \{-1,+1\}$. For  $\varphi : E \rightarrow \R$ we often write $\varphi^+(y) := \varphi(y,+1)$ and $\varphi^-(y) := \varphi(y,-1)$. If we write $\varphi^{\pm}$ we mean both $\varphi^+$ and $\varphi^-$. Equip $E$ with the product topology and let $C(E)$ denote the space of continuous functions $\varphi : E \rightarrow \R$. Note that $\varphi \in C(E)$ if and only if $\varphi^{\pm} \in C(\R)$. Let $C_0(E)$ denote the linear subspace of $\varphi \in C(E)$ which vanish at infinity, i.e. $\varphi^{\pm} \in C_0(\R)$ (where $C_0(\R)$ denotes the Banach space of continuous functions on $\R$, vanishing at infinity). Let $C^1(\R)$ denote the space of continuously differentiable functions on $\R$.

Throughout this section let $\lambda : E \rightarrow [0,\infty)$ be continuous, and $a > 0$.
Introduce a densely defined linear operator on $C_0(E)$
\begin{equation}
\label{eq:abstract-generator}
 L \varphi(y,j) = a j \frac{\partial \varphi}{\partial y}(y,j) + \lambda(y,j) (\varphi(y,-j) - \varphi(y,j)), \quad y \in \R, j = \pm 1,
\end{equation}
with domain $D(L) = \{ \varphi : E \rightarrow \R, \varphi^{\pm} \in C^1(\R), (L \varphi)^{\pm} \in C_0(\R)\}$. It is easy to verify that $L$ is closable.

\subsection{Construction of the zig-zag process}
\label{sec:construction}

\begin{assumption}
\label{ass:construction}
There exist constants $y_0 \geq 0$ and $\lambda_{\min} > 0$ such that $\lambda (y, j) \geq \lambda_{\min}$ for $j y \geq y_0$.
\end{assumption}

We will call a switch from the $(j)$-replica to the $(-j)$-replica a `good switch' when $jy \geq y_0$, and a `bad switch' when $jy \leq -y_0$. For example, a switch from $+1$ to $-1$ is good for $y \geq y_0$, but bad for $y \leq - y_0$. Good switches make the process direct itself towards the origin, whereas bad switches do the opposite. If `too few'  good switches occur, the process might wander off to infinity. Assumption~\ref{ass:construction} states that for $|y| \geq y_0$ there is a lower bound for the rate at which good switches occur.

For $(y,j) \in E$, define the survival function 
\begin{equation}
\label{eq:survival} F(t;y,j) := \exp \left( - \int_0^t \lambda( y + a j s, j) \ d s \right), \quad t \geq 0.
\end{equation}
Since $\lambda$ is continuous, and hence bounded on compact sets, for every $(y,j) \in E$ and $t \geq 0$, $F(t;y,j) > 0$.
It is established in Lemma~\ref{lem:finite-switching-times} that for every $(y,j) \in E$, $1 - F(\cdot, y,j)$ is the distribution function of a strictly positive random variable that is almost surely finite. In fact, $1- F(\cdot, y, j)$ will serve as the distribution of the random time at which the value of $j$ will be switched, starting from $(y,j)$.

Given $(y,j) \in E$, define the process $(Y(t),J(t))$ along with random variables $(Z_i)_{i \in \{1, 2, \dots\}}$ and $(T_i)_{i \in \{0,1, 2, \dots\}}$ as follows. 
\begin{itemize}
 \item Let $T_0 = 0$, $J(0) = j$, $Y(0) =y$.
 \item For $i = 1, 2, \dots$
 \begin{itemize}
 \item Let $Z_i$ be distributed according to $\P_{y,j}(Z_i > t \mid Z_1, \dots, Z_{i-1}) = F(t; Y(T_{i-1}), J(T_{i-1}))$;
 \item Let $T_i := T_{i-1} + Z_i$;
 \item Define $J(t) = J(T_{i-1})$ for $T_{i-1} < t < T_i$ and $J(T_i) = -J(T_{i-1})$;
 \item Define $Y(t) = Y(T_{i-1}) + J(T_{i-1}) a (t-T_{i-1})$ for $T_{i-1} < t \leq T_i$
\end{itemize}
\end{itemize}
Then $T_k := \sum_{i =1}^k Z_i$. The process $(Y(t))$ is continuous and piecewise linear, and $(J(t))$ is piecewise constant and right-continuous. It follows that $(Y(t),J(t))_{t \geq 0}$ is cadlag. For $t \geq 0$ let 
\[ N(t) := \sup \left\{ k \in \N : T_k \leq t\right\},\]
the number of switches that have occured up to  time $t$. 
We have defined $(Y(t),J(t))$ up to $t < T_{\infty} := \lim_{k \rightarrow \infty} T_k \leq \infty$. 
By Lemma~\ref{lem:finitely-many-switches}, we can exclude the possibility that $\lim_{k \rightarrow \infty} T_k < \infty$.

Let $\P_{y,j}$ denote the probability distribution conditional over these random variables given that $Y(0) = y, J(0) = j$.
Let $\mathcal F_t := \sigma( \{ (Y(s),J(s)): s \leq t \})$.

\begin{proposition}
\label{prop:construction}
Suppose Assumption~\ref{ass:construction} holds. Then under $\P_{y,j}$, the process $(Y,J)$ is a non-explosive strong Markov process with respect to $(\mathcal F_t)$, with generator equal to the closure of~\eqref{eq:abstract-generator}.
\end{proposition}

\begin{proof}
This follows directly from general theory for piecewise deterministic Markov processes; see \cite{Davis1984}.
\end{proof}

\subsection{Regularity}

Let $P = (P(t))_{t \geq 0}$ denote the Markov semigroup corresponding to the zig-zag process $(Y,J)$.
By a coupling argument, we can establish the Feller property for $P$. The Feller property of piecewise deterministic Markov processes is established in \cite{Davis1993} under the assumption of bounded switching rates, which is not satisfactory in our setting.
The proofs of this proposition and subsequent results are located in Section~\ref{sec:proof-zigzag}. 

\begin{proposition}
\label{prop:feller}
Suppose Assumption~\ref{ass:construction} holds. The Markov transition semigroup $P$ with infinitesimal generator $L$ is Feller, i.e. for every $\varphi \in C_0(E)$ and $t \geq 0$, we have
$P(t) \varphi \in C_0(E)$.
\end{proposition}

Let $B_b(E)$ and $C_b(E)$ denote the sets of bounded Borel measurable functions and bounded continuous functions on $E$, respectively. Recall that 
$(P(t))_{t \geq 0}$ is strong Feller if $P(t)\varphi \in C_b(E)$ for any $t > 0$ and any $\varphi \in B_b(E)$. The transition  semigroup corresponding to the zig-zag process does not satisfy this property. 

\begin{observation}
\label{obs:not-strong-Feller}
Suppose Assumption~\ref{ass:construction} holds. Then $(P(t))_{t \geq 0}$ is not strong Feller.
\end{observation}

\begin{proof}
Let $j = +1$ and $y \in \R$. Let $t > 0$ and let $A = [y + a t, \infty)$. Let $\varphi(y,j) = \1_{A}(y)$.
Because $t < T_1$ implies $Y(t) \in A$, it follows that
$P(t) \varphi(y,j) = \P_{y,j}(Y(t) \in A) \geq \P_{y,j}(T_1 > t) > 0$. However $P(t) \varphi(z,j) = \P_{z,j}(Y(t) \in A) = 0$ for every $z < y$, 
so that $P(t) \varphi$ is not continuous.
\end{proof}

\subsection{Invariant measure}

Let us strengthen Assumption~\ref{ass:construction} into the following assumption.
\begin{assumption}
\label{ass:invariant-measure}
There exist constants $y_0 \geq 0$ and $\lambda_{\min} > 0$ such that
\begin{itemize}
 \item[(i)] $\lambda (y, j) \geq \lambda_{\min}$ for $j y \geq y_0$, and
 \item[(ii)] $\lambda(y,-j) \leq \lambda(y,j)$ for $j y \geq y_0$.
\end{itemize}
\end{assumption}

We strengthened Assumption~\ref{ass:construction} by requiring that in the tails the rate at which `good switches' (i.e. mean reverting switches) occur is higher than the rate of `bad switches'.

\begin{proposition}
\label{prop:invariant-measure}
Suppose Assumption~\ref{ass:invariant-measure} holds. Let $\Psi : \R \rightarrow \R$ be defined by
\[ \Psi(y) = \frac 1 a \int_{0}^y \{\lambda^+(\eta) - \lambda^-(\eta)\} \ d \eta, \quad y \in \R.\]
Then $\Psi$ is bounded from below, $\Psi(y) <\infty$ for all $y \in \R$, and the Markov process $(Y,J)$ has invariant measure $\mu$ with density $(y,j) \mapsto \exp(-\Psi(y))$ with respect to $\mathrm{Leb} \otimes (\delta_{-1} + \delta_{+1})$ on $E$.
\end{proposition}

Under the stated assumption we can not yet make any claims as to whether $\mu$ is a finite measure. As an example consider the case in which $\lambda(y, \pm j) = \lambda_0 > 0$ for all $(y,j)$, which satisfies Assumption~\ref{ass:invariant-measure}. By Proposition~\ref{prop:invariant-measure} this corresponds to a uniform invariant density.

The proof of Proposition~\ref{prop:invariant-measure} is a simple computation that we will include here.

\begin{proof}[Proof of Proposition~\ref{prop:invariant-measure}]
Note that $\Psi$ and $\lambda$ are related by
\begin{equation}
\label{eq:inv_measure}
 a  \frac{d\Psi(y)}{d y} + \lambda^-(y)  - \lambda^+(y) = 0, \quad y \in \R.
\end{equation}
It follows from Assumption~\ref{ass:invariant-measure} that $\Psi$ is bounded from below and $\Psi(y) < \infty$ for all $y \in \R$. 
Suppose $\varphi \in D(L)$ and suppose $\mu$ is as specified. 
Then, using that $\Psi$ is bounded from below and $\varphi \in C_0(E)$ in the partial integration below,
\begin{align*}
 & \sum_{j = -1, +1} \int_{-\infty}^{\infty} L \varphi(y,j) \ d \mu(y,j) \\
 & = \sum_{j = -1, +1}\int_{-\infty}^{\infty} \left\{ a j \frac{\partial \varphi(y,j)}{\partial y} + \lambda(y,j) (\varphi(y,-j) - \varphi(y,j))\right\}  \exp(-\Psi(y)) \ d y \\
 & = \sum_{j=-1,+1}\int_{-\infty}^{\infty} \left\{  a j \frac{d \Psi(y)}{d y} -  \lambda(y,j) \right\} \varphi(y,j)  \exp(-\Psi(y)) \ d y \\
 & \quad + \sum_{k = -1,+1} \int_{-\infty}^{\infty} \lambda\left(y, - k \right) \varphi\left(y, k \right) \exp\left(- \Psi(y)\right) \ d y \\
 & = \sum_{j=-1,+1} \int_{-\infty}^{\infty} \left\{ a j \frac{d \Psi(y)}{d y} - \lambda(y,j) + \lambda(y,-j)  \right\} \varphi(y,j) \exp(-\Psi(y))  \ d y = 0.
\end{align*}
Note that we first let $k = -j$ and in the next step replaced $k$ by $j$.
It follows that 
\[ \sum_{j=-1,+1} \int_{-\infty}^{\infty} P(t) \varphi(y,j) \ d \mu(y,j) = \mu(\varphi), \quad \varphi \in D(L), t \geq 0.\]
By a standard approximation argument, this holds for any $\varphi \in B_b(E)$, and it follows that $\mu$ is invariant for $P$.
\end{proof}

\subsection{Exponential ergodicity}
We will further strengthen Assumption~\ref{ass:invariant-measure} into the following assumption, which therefore also implies Assumption~\ref{ass:construction}.

\begin{assumption}
\label{ass:lyapunov-improved}
There is a $y_0 > 0$ such that
\begin{itemize}
 \item[(i)] $\inf_{y \geq y_0} \lambda^+(y) >  \sup_{y \geq y_0} \lambda^-(y)$, and
 \item[(ii)]  $\inf_{y \leq -y_0} \lambda^-(y) > \sup_{y \leq -y_0} \lambda^+(y)$.
\end{itemize}
\end{assumption}

\begin{lemma}[Invariant measure is finite]
\label{lem:invariant_probability_measure}
Suppose Assumption~\ref{ass:lyapunov-improved} holds and $\Psi$ satisfies~\eqref{eq:inv_measure}. Then $\mu$ defined in Proposition~\ref{prop:invariant-measure} is finite, i.e. $\mu(E) < \infty$.
\end{lemma}

\begin{proof}
Using Assumption~\ref{ass:lyapunov-improved}, we have $\lambda^+(y) - \lambda^-(y) \geq c$ on $[y_0, \infty)$ for some $c > 0$. Therefore
\[ \int_{y_0}^{\infty} \exp(-\Psi(y)) \ dy \leq \int_{y_0}^{\infty} \exp\left(-\Psi(y_0) - c  (y-y_0) \right) \ d y < \infty,\]
and similarly for the integral over $(-\infty,-y_0]$.
\end{proof}

Without loss of generality, we will assume below that $\mu$ is a probability measure, i.e. $\mu(E) = 1$.
For $f : E \rightarrow [1, \infty)$ define the $f$-norm by
\[ \| \mu \|_{f} = \sup_{|g| \leq f} |\mu(g)|, \quad \mu \ \mbox{signed measure on} \ \mathcal B(E),\]
which is a stronger norm than the total variation norm. By characterizing the `petite sets'  and using a Foster-Lyapunov function (Lemmas~\ref{lem:compact-sets-are-petite} and Lemma~\ref{lem:lyapunov}, respectively, located in Section~\ref{sec:proofs}), we can establish exponential ergodicity. We acknowledge the recommendation of a referee to use the Lyapunov function of \cite{Fontbona2015} instead of our earlier construction, which allowed us to further weaken the conditions under which we obtain exponential ergodicity.  A function $V \in C(E)$ is norm-like if $\lim_{|x| \rightarrow \infty} V(x) = \infty$.

\begin{theorem}
\label{thm:exponential-ergodicity}
Suppose Assumption~\ref{ass:lyapunov-improved} holds. Then $(Y(t),J(t))_{t \geq 0}$ is exponentially ergodic, i.e. there exist constants $0 < \rho < 1$ and $0 < \kappa < \infty$ and a norm-like function $V$ such that
\[ \| \P_{y,j}((Y(t),J(t)) \in \cdot) - \mu \|_{f} \leq \kappa f(y,j) \rho^t, \quad t \geq 0,\]
where $f(y,j) = 1 + V(y,j)$.
\end{theorem}

\begin{proof}
By Lemma~\ref{lem:compact-sets-are-petite} and Lemma~\ref{lem:lyapunov}, all conditions of \cite[Theorem 6.1]{MeynTweedie1993-III} are satisfied, so that the stated result follows.
\end{proof}

\subsection{Application to Curie-Weiss}

In the Curie-Weiss model, the generator obtained in Theorems~\ref{thm:LMH-supercritical-temperature} and \ref{thm:LMH-critical-temperature} is given by~\eqref{eq:abstract-generator} with
$a = a(h,\beta)$ given by~\eqref{eq:LMH-drift} and $\lambda(y,j) = \max\left(0, j \frac{d \Psi(y)}{d y} \right)$, with
\[ \Psi(y) = \left\{ \begin{array}{ll} y^4/12 \quad & \beta = 1, h = 0, \\
                      \half l(h,\beta) y^2, \quad & 0 \leq \beta < 1, h \in \R,
                     \end{array} \right.\]
with $l(h, \beta)$ given by~\eqref{eq:drift-high-temperature}. In particular $\lambda(y,j) > 0$ for $j y > 0$ and $\lambda(y,j) = 0$ for $j y \leq 0$. It follows that Assumption~\ref{ass:lyapunov-improved} is satisfied, taking any $y_0 >0$. Assumptions~\ref{ass:construction} and~\ref{ass:invariant-measure} are weaker than Assumption~\ref{ass:lyapunov-improved}. To summarize, we have the following corollary.

\begin{corollary}
$L$ given by~\eqref{eq:abstract-generator}, with $a$ and $\lambda(y,j)$ as above, is the generator of a Markov-Feller transition semigroup on $C_0(E)$. The associated Markov process $(Y,J)$ has finite invariant measure $\mu$ on $E$ as in Proposition~\ref{prop:invariant-measure} and is exponentially ergodic.
\end{corollary}

\begin{proof}
This is a combination of Propositions~\ref{prop:construction},~\ref{prop:feller},~\ref{prop:invariant-measure} and Theorem~\ref{thm:exponential-ergodicity}.
\end{proof}

\section{Proofs}
\label{sec:proofs}

\subsection{Estimates on Metropolis-Hastings applied to Curie-Weiss}

We can easily compute the difference in interaction energy for increments in $\eta$,
\begin{equation} \label{eq:phi-difference} \Phi^n(\eta) - \Phi^n(\eta \pm 2 n^{\gamma - 1}) = \pm 2 n^{-\gamma} \eta \pm 2 (m_0+h) + 2 n^{-1}.
\end{equation}
Combined with~\eqref{eq:transition-probs-eta-rw} and~\eqref{eq:MH-definition}, it follows that
\begin{equation}
\begin{aligned}
 \label{eq:MH-transitions} p^n_{\pm}(\eta) & := P^n(\eta, \eta \pm 2 n^{\gamma - 1}) \\
 & = \half ( 1 \mp (m_0 + n^{-\gamma} \eta)) \left( 1 \wedge \exp \left( \beta \{ \pm 2 n^{-\gamma} \eta \pm 2 (m_0+h) + 2 n^{-1} \} \right) \right).
 \end{aligned}
\end{equation}
Due to the possibility of rejection, there will be positive mass on transition probabilities $P^n(\eta, \eta)$. These values are fully determined by the off-diagonal transition probabilities and will not appear in the analysis below.

To rephrase slightly, for $\eta \in X^n$, define probability distributions $\P^n_{\eta}$ on $X^n$, and let $Y$ denote $X^n$-valued random variables with distribution $P^n(\eta, \cdot)$. In other words, under $\P^n_{\eta}$, $Y$ is distributed according to $P^n(\eta, \cdot)$. Expectation with respect to $\P^n_{\eta}$ will be denoted by $\E_{\eta}^n$, so that $\E_{\eta}^n[ \varphi(Y)] = P^n \varphi(\eta)$ for $\varphi : X^n \rightarrow \R$.

We will be particularly interested in values of $\eta$ that are concentrated on the following sets
\begin{equation} \label{eq:definition-F} F^{n,\delta} := \{ \eta \in X^n : |\eta| \leq n^{\delta} \},\end{equation}
where $\delta < \gamma$.
In the computations that follow, we will frequently need to approximate the exponent in the Metropolis-Hastings acceptance probability by its Taylor approximation. The following lemma helps in determining the required order of approximation.
Let $p_k(x)$ denote the $k$-th order Taylor approximation of $\exp(x)$, i.e. $p_k(x) = \sum_{i=0}^k \frac {x^i}{i !}$. Define approximate transition probabilities
\begin{equation}
\label{eq:simplified-transitions}
\begin{aligned}
 p^{n,k}_{\pm}(\eta) :=  & \half (1 \mp(m_0 + n^{-\gamma} \eta)) \\
 & \hfill \times \left\{ \begin{array}{ll}  1 \quad & \mbox{if $h +m_0 =0$ and $\pm \eta \geq 0$,} \\
                                 p_k \left( \pm 2 \beta n^{-\gamma} \eta \right) \quad & \mbox{if $h+m_0 =0$ and $\pm \eta < 0$,} \\
                                \exp \left( \pm 2 \beta (m_0 + h) \right) p_k \left(\pm 2 \beta n^{-\gamma} \eta \right), \quad & \mbox{if $\pm (h + m_0) < 0$}, \\
                                1 \quad & \mbox{if $\pm(h+m_0) > 0$}.
                               \end{array}\right.
\end{aligned}
\end{equation}
For example, if $\eta < 0$, and $h+m_0 = 0$, then
\[ p^{n,k}_+(\eta) = \half (1 - (m_0 + n^{-\gamma} \eta)) p_k(2 \beta n^{-\gamma} \eta).\]

\begin{lemma}
\label{lem:simplify}
Let $0 < \delta < \gamma < 1$. Suppose $h \neq 0$, $\beta > 0$ or $h = 0$, $0 \leq \beta \leq 1$. 
Then, for $r < \min(1, (k+1) (\gamma-\delta))$,
\[ \lim_{n \rightarrow \infty} \sup_{\eta \in F^{n, \delta}} n^r \left|p^n_{\pm}(\eta) - p^{n,k}_{\pm}(\eta)\right|  \rightarrow 0. \]
\end{lemma}

\begin{proof}
The result is trivial in case $\beta = 0$. In the remainder therefore assume $\beta > 0$. Define
\[ \widetilde p^n_{\pm}(\eta) := \half ( 1 \mp (m_0 + n^{-\gamma} \eta)) \left( 1 \wedge \exp \left(  \pm  2 \beta \{n^{-\gamma} \eta + (m_0+h)\} \right) \right).\]
(This is just $p^n_{\pm}(\eta)$ without the $O(n^{-1})$ term in the exponent.)
We estimate, using 1-Lipschitz continuity of $x \mapsto 1 \wedge e^{-x}$ (for $x \geq 0$), $|m_0 + n^{-\gamma} \eta | \leq 1$, and $r < 1$,
\begin{align*}
\sup_{\eta \in F^{n, \delta}} n^r \left|p^n_{\pm}(\eta) - 
\widetilde p^n_{\pm}(\eta)\right| & \leq \sup_{\eta \in F^{n,\delta}} \half | 1 \mp (m_0 + n^{-\gamma} \eta)| 2 \beta n^{r-1} \leq 2 \beta n^{r-1} \rightarrow 0
\end{align*}
as $n \rightarrow \infty$.
For $k \in \N$, define further approximate transition probabilities
\[\widetilde p^{n,k}_{\pm}(\eta)= \half ( 1 \mp (m_0 + n^{-\gamma} \eta)) \min \left\{ 1, \exp\left(\pm 2 \beta (m_0 + h) \right) p_k\left(\pm 2 \beta n^{-\gamma} \eta \right) \right\}.\]
Then, using 1-Lipschitz continuity of $x \mapsto 1 \wedge x$,  $|n^{-\gamma} \eta| \leq n^{\delta -\gamma} \leq 1$ on $F^{n,\delta}$, 
and $|p_k(x) - \exp(x)| \leq e^{\xi} \frac{x^{k+1}}{(k+1)!}$ for some $\xi \in (\min(0, x), \max(0,x))$, it follows that
\begin{align*}
\sup_{\eta \in F^{n,\delta}} n^r \left| \widetilde p^n_{\pm}(\eta) - \widetilde p^{n,k}_{\pm}(\eta) \right| & \leq n^r  \exp(\pm 2 \beta (m_0 +h) ) \left| \exp(\pm 2 \beta n^{-\gamma} \eta) - p_k(\pm 2 \beta n^{-\gamma} \eta) \right| \\
& \leq n^r \exp(\pm 2 \beta (m_0 +h) ) \exp( 2 \beta )\frac {(2 \beta n^{-\gamma} \eta)^{k+1}}{(k+1)!} \\
& = c n^{r  -(k+1) (\gamma-\delta)} \rightarrow 0
\end{align*}
as $n \rightarrow \infty$.
In the limit as $n \rightarrow \infty$, the minimization in the expression for $\widetilde p^{n,k}_{\pm}(\eta)$ will only depend on the lowest order terms. Since the convergence of $n^{-\gamma} \eta$ is uniform on $F_{n,\delta}$, the stated result follows after distinguishing cases for $h + m_0 = 0$ and $h + m_0 \neq 0$.
\end{proof}

As a first example of the use of Lemma~\ref{lem:simplify}, we have the following result for the second moment of Metropolis-Hastings updates. We introduce a multiplicative factor $n^{\alpha}$ which will represent speeding up the Markov chain: within a time interval of length $t \in \R$ we will make $N(t)$ switches according to $P^n$, where $N(t) \sim \mathrm{Poisson}(n^{\alpha} t)$. One of the results of our analysis is the correct value of $\alpha$ for which a suitable scaling limit is obtained, which turns out to be related to $\gamma$ by $\alpha = 2(1-\gamma)$.

\begin{lemma}[Metropolis-Hastings second moment for Curie-Weiss]
\label{lem:diffusion-general-temperature}
Let $0 < \delta < \gamma < 1$. Let $\alpha = 2(1-\gamma)$. Suppose $h \neq 0$, $\beta > 0$ or $h = 0$, $0 \leq \beta \leq 1$. 
Define
\begin{equation}
 \label{eq:diffusion-general-temperature} \sigma(h,\beta) := 2 \sqrt{1 - |m_0(h,\beta)|}.
\end{equation}
Then 
\begin{equation} \label{eq:convergence-second-moment-mh} \lim_{n \rightarrow \infty} \sup_{\eta \in F^{n,\delta}} |n^{\alpha} \E_{\eta}^n[(Y-\eta)^2] - \sigma(h,\beta)^2| = 0.\end{equation}
\end{lemma}

\begin{proof}
We have 
\begin{align*}
 n^{\alpha} \E_{\eta}^n[(Y-\eta)^2] =  n^{\alpha} (p^n_+(\eta) + p^n_-(\eta)) (2 n^{\gamma - 1})^2 = 4  (p^n_+(\eta) + p^n_-(\eta)).
\end{align*}
We may apply Lemma~\ref{lem:simplify} with $r = 0$ and $k = 0$, to deduce that
\[ \lim_{n \rightarrow \infty} \sup_{\eta \in F^{n,\delta}} \left| n^{\alpha} \E[(Y-\eta)^2]  - 4(p^{n,0}_+(\eta) + p^{n,0}_-(\eta)) \right| = 0,\]
where
\begin{align*}& p^{n,0}_+(\eta) + p^{n,0}_-(\eta) \\
& = \half (1 -m_0 - n^{-\gamma} \eta) \left( 1 \wedge \exp(2 \beta (m_0+h)) \right) + \half (1 + m_0 + n^{-\gamma} \eta) \left( 1 \wedge \exp(-2 \beta (m_0+h)) \right).\end{align*} 
On $F^{n,\delta}$ we have $|n^{-\gamma} \eta| \leq n^{\delta- \eta} \rightarrow 0$, so the remaining dependence on $n$ in the above expression vanishes asymptotically, and we conclude that~\eqref{eq:convergence-second-moment-mh} holds for
\[\sigma(h,\beta) = \left\{ 2(1-m_0) \left( 1 \wedge \exp(2 \beta (m_0+h)) \right) + 2 (1 + m_0) \left( 1 \wedge \exp(-2 \beta (m_0+h)) \right) \right\}^{1/2}.\]
Distinguishing cases and using~\eqref{eq:curie-weiss}, this is equal to the stated expression for $\sigma$.\end{proof}

Another useful observation is that higher order moments of $Y - \eta$ vanish:

\begin{lemma}
\label{lem:higher-moments}
Let $0 < \delta < \gamma < 1$ and let $\alpha = 2(1-\gamma)$. Suppose $h \neq 0$, $\beta > 0$ or $h = 0$, $0 \leq \beta \leq 1$. Then
\[ \lim_{n \rightarrow \infty} \sup_{\eta \in F^{n,\delta}} n^{\alpha} \E^n_{\eta}[|Y-\eta|^p] = 0 \]
for any $p > 2$.
\end{lemma}
\begin{proof}
We have 
\begin{align*} \lim_{n \rightarrow \infty} \sup_{\eta \in F^{n,\delta}} n^{\alpha} \E^n_{\eta}[|Y-\eta|^p] & = \lim_{n \rightarrow \infty} \sup_{\eta \in F^{n,\delta}} 2^p  (p^n_+(\eta) + p^n_-(\eta))  n^{(2 -p)(1 - \gamma)}  = 0,
\end{align*}
using that the sum of the probabilities is bounded by 1.
\end{proof}

\subsubsection{Supercritical temperature regime}

We already mentioned that in the supercritical temperature case ($0 \leq \beta < 1$), the correct scaling of the magnetization would be $\gamma = \half$. 

\begin{lemma}
\label{lem:concentration-high-temperature}
Suppose $\gamma = \half$ and $\delta \in (0, \half)$. If $0 \leq \beta < 1$, then for any $\alpha > 0$,
\[ \lim_{n \rightarrow \infty} n^{\alpha} \pi^n\left(\eta^n(x) \notin F^{n,\delta}\right) = 0.\]
\end{lemma}

\begin{proof}
Note
\[ \{ x : \eta^n(x) \notin F^{n, \delta} \} = \{x : |m^n(x) -m_0| > n^{\delta - \gamma}\}.\]
By the mean value theorem, $|m - \tanh(\beta(m+h))| \geq (1 -\beta) |m-m_0|$ for $m \in \R$.
Therefore, using Proposition~\ref{prop:concentration-Curie-Weiss}, with $t_n := ((1-\beta) n^{\delta - \gamma} - \beta n^{-1}) n^{1/2}$, we find that
\begin{align*}
\pi^n(|m^n(x) - m_0| \geq n^{\delta - \gamma}) & \leq \pi^n(|m - \tanh \beta(m+h)| \geq (1-\beta) n^{\delta - \gamma}) \\ & = \pi^n\left(|m - \tanh \beta(m+h)| \geq \frac{\beta}{n} + \frac{t_n}{\sqrt{n}}\right)  \leq 2 \exp\left(-\frac{t_n^2}{4(1+\beta)}\right) 
\end{align*}
from which the result follows, using that $t_n \sim n^{1/2 + \delta - \gamma} \rightarrow \infty$.
\end{proof}

\begin{lemma}
\label{lem:drift-high-temperature}
Suppose $0 \leq \beta < 1$, $\gamma = \half$, $\alpha = 2(1-\gamma) = 1$ and $\delta \in (0, 1/4)$. Let
\begin{equation}
\label{eq:drift-high-temperature}
l(h,\beta) := \frac{1}{1 + |m_0|} -\beta(1 - |m_0|)
\end{equation}
Then $\lim_{n \rightarrow \infty} \sup_{\eta \in F^n} |n^{\alpha} \E^n_{\eta}[ Y - \eta] + 2 l(h,\beta) \eta | = 0$.
\end{lemma}

In other words, the ``drift'' function of the Metropolis-Hastings transitions is given by $-2 l(h,\beta) \eta$.

\begin{proof}
We have, using $\alpha = 1$,
\[ n^{\alpha} \E^n_{\eta}[Y-\eta] = n (p^n_+(\eta) - p^n_-(\eta)) (2 n^{-1/2}).\]
Therefore, applying Lemma~\ref{lem:simplify} with $r = \half$ and $k = 1$, we may approximate $n^{\alpha} \E^n_{\eta}[Y-\eta]$ to sufficient precision by 
\[ 2 n^{1/2}(p^{n,1}_+(\eta) - p^{n,1}_-(\eta)).\]
Now distinguish the following cases.
\begin{itemize}
 \item Suppose $m_0 = 0$ and (therefore) $h+m_0 = 0$. We will show the result for $\eta \geq 0$, the case $\eta < 0$ is analogous. If $\eta \geq 0$, then
\begin{align*}
2 n^{1/2} (p^{n,1}_+(\eta) - p^{n,1}_-(\eta)) & = n^{1/2} (1 - n^{-1/2} \eta) - n^{1/2}(1 + n^{-1/2} \eta)(1 - 2 \beta n^{-1/2} \eta)  \\
& = -2 (1 - \beta) \eta + 2 \beta n^{-1/2} \eta^2.
\end{align*}
Now using that $n^{-1/2}\eta^2 \leq n^{2\delta - 1/2} \rightarrow 0$ on $F^{n,\delta}$, the result follows.
\item Suppose $h + m_0 > 0$ (the case $h + m_0 < 0$ is analogous). Then, using~\eqref{eq:curie-weiss},
\begin{align*}
& 2 n^{1/2} (p^{n,1}_+(\eta) - p^{n,1}_-(\eta)) \\
& = n^{1/2}(1 - (m_0 + n^{-1/2} \eta)) - n^{1/2}(1 + m_0 + n^{-1/2}\eta) \exp(- 2 \beta (m_0+h)) (1 - 2 \beta n^{-1/2} \eta)\\
& = n^{1/2} (1 - (m_0 + n^{-1/2}\eta)) - n^{1/2}(1 + m_0 + n^{-1/2} \eta) \left( \frac{1-m_0}{1+m_0} \right) (1 - 2 \beta n^{-1/2} \eta) \\
& = -\eta - \eta \left( \frac{1 - m_0}{1 + m_0} \right) + 2 \beta  (1 - m_0) \eta + 2 \beta \left( \frac{1 - m_0}{1 + m_0} \right) n^{-1/2} \eta^2 \\
& = -\frac {2 \eta}{1 + m_0} + 2 \beta (1 - m_0) \eta + 2 \beta \left( \frac{1 - m_0}{1 + m_0} \right) n^{-1/2} \eta^2,
\end{align*}
where again the $O(n^{-1/2}\eta^2)$-term vanishes.
\end{itemize}
\end{proof}

\begin{proof}[Proof of Theorem~\ref{thm:diffusion-limit-mh-high-temperature}]
The generator of $Y^n$ is given by 
\begin{equation} \label{eq:generator-markov-chain} G^{n, \alpha} \varphi(\eta) := n^{\alpha} (P^n \varphi(\eta ) - \varphi(\eta)), \quad \varphi : X^n \rightarrow \R.\end{equation}
Let $G$ denote the unbounded operator $G : D(G) \subset C_0(\R) \rightarrow C_0(\R)$, where
\[ D(G) = \{ \varphi \in C^2_0(\R) : \eta \mapsto \eta \varphi'(\eta)  \in C_0(\eta) \}\]
and
\[ G\varphi(\eta) = - 2 l(h, \beta)\eta \frac{d \varphi}{d \eta} + \half \sigma^2(h,\beta) \frac{d^2 \varphi}{d \eta^2}, \quad \varphi \in D(G).\]
The space of infinitely differentiable functions with compact support $C^{\infty}_c(\R)$ is strongly separating (in the sense of \cite[Section 3.4]{EthierKurtz2005}). Let $\varphi \in C^{\infty}_c(\R)$. For $\eta, \zeta \in \R$, we have
\[ \left| \varphi(\zeta) - \left\{ \varphi(\eta) + \varphi'(\eta)(\zeta - \eta) + \half \varphi''(\eta)(\zeta - \eta)^2 \right\} \right| \leq (1/6) \| \varphi^{(3)}\|_{\infty}|\zeta - \eta|^3.\]
 Since $\varphi^{(3)}$ is bounded we may approximate, for $\alpha = 1$ and $\delta = 1/8$, using Lemmas~\ref{lem:diffusion-general-temperature},~\ref{lem:higher-moments} and ~\ref{lem:drift-high-temperature},
\begin{align*}
& \sup_{\eta \in F^{n,\delta}} \left| G^{n,\alpha} \varphi(\eta) - G \varphi(\eta)\right| \\
& \leq \sup_{\eta \in F^{n,\delta}} \left| n^{\alpha} \E^n_{\eta} [ \varphi_n(Y) - \varphi_n(\eta)] - n^{\alpha} \E^n_{\eta}[ \varphi'(\eta) (Y - \eta) + \half \varphi''(\eta) (Y- \eta)^2] \right| \\
& \leq (1/6) n^{\alpha}\|\varphi^{(3)}\|_{\infty} \E_{\eta}^n [|Y- \eta|^3]\rightarrow 0.\end{align*}

Let $\P^n$ denote the distribution of the stationary Markov process $Y^n$ with invariant distribution $\mu^n$. Then, for $T > 0$, by Lemma~\ref{lem:concentration-high-temperature},
\[ \P^n(Y^n(t) \notin F^{n,\delta} \ \mbox{for some $0 \leq t \leq T$}) \leq n^{\alpha} \pi^n(\eta^n(x) \notin F^{n,\delta}) \rightarrow 0.\] 
We may now apply \cite[Corollary 4.8.7]{EthierKurtz2005} to arrive at the stated result.
\end{proof}

\subsubsection{At critical temperature}
In this section we assume the `critical' case $h = 0$ and $\beta = 1$. The correct scaling of the magnetization will be 
$\eta^n = n^{\gamma - 1} m^n$ with $\gamma = 1/4$.

\begin{lemma}
\label{lem:concentration-critical-temperature}
Suppose $h = 0$ and $\beta = 1$. Let $\gamma = 1/4$ and $\delta > 0$. Then, for any $\alpha > 0$, 
\[ \lim_{n \rightarrow \infty} n^{\alpha} \pi^n( \eta^n(x) \notin F^{n,\delta}) = 0.\]
\end{lemma}
\begin{proof}
This follows since, by Proposition~\ref{prop:concentration-Curie-Weiss} (ii),
\[ n^{\alpha} \pi^n(\eta^n \notin F^n) = n^{\alpha}\pi^n(|m^n| > n^{\delta - \gamma}) \leq 2 n^{\alpha} \exp\left( - c n^{1 + 4(\delta - \gamma)}\right) \rightarrow 0.\]
\end{proof}

It turns out that in this case, the correct speed-up factor is $n^{\alpha}$ with $\alpha = 3/2$. In order to obtain the generator in the critical regime, we will require higher-order Taylor expansions, resulting in a non-linear drift in the diffusion limit, and accordingly, a non-Gaussian invariant distribution.

\begin{lemma}
\label{lem:drift-critical-temperature}
Suppose $\beta = 1$ and $h = 0$. Let $\delta \in (0, 1/16)$. Then 
\[ \lim_{n \rightarrow \infty} \sup_{\eta \in F^{n,\delta}} \left| n^{3/2} \E[Y^n - \eta] + 2/3 \eta^3  \right| = 0. \]
\end{lemma}
\begin{proof}
We have $n^{3/2} \E^n_{\eta}[Y- \eta] = n^{3/2} (p^n_+(\eta)-p^n_-(\eta)) (2n^{-3/4}) = 2 (p^n_+(\eta) - p^n_-(\eta)) n^{3/4}$. Applying Lemma~\ref{lem:simplify} with $r=3/4$ and $\gamma = 1/4$, we find that we may approximate $p^n_{\pm}$ by the 3-rd order approximation $p^{n,3}_{\pm}$. Assuming $\eta > 0$ (the other case is analogous),
\begin{align*}
&  2 (p^{n,3}_+(\eta) - p^{n,3}_-(\eta)) n^{3/4} \\
& = \left( (1 - n^{-1/4} \eta) - (1 + n^{-1/4}\eta)(1 - 2 n^{-1/4} \eta  + 2 n^{-1/2} \eta^2 - (4/3) n^{-3/4} \eta^3) \right)n^{3/4} \\
 & = \left(-(2/3) n^{-3/4} \eta^3 + (4/3)n^{-1} \eta^4 \right) n^{3/4}.
\end{align*}
On $F^{n,\delta}$ with $\delta < 1/16$, we have $\eta^4 \leq n^{4 \delta} < n^{1/4}$. It follows that the fourth order term in $\eta$ vanishes asymptotically, and the stated result follows.
\end{proof}

\begin{proof}[Proof of Theorem~\ref{thm:diffusion-limit-mh-critical-temperature}]
The proof is completely analogous to that of Theorem~\ref{thm:diffusion-limit-mh-high-temperature}, taking $\alpha = 3/2$, $\gamma = 1/4$, $\delta = 1/32$, and applying Lemma~\ref{lem:drift-critical-temperature} instead of Lemma~\ref{lem:drift-high-temperature}.\end{proof}

\subsection{Estimates for Lifted Metropolis-Hastings applied to Curie-Weiss}

Let $(Y,J) \in \R \times \{-1,1\}$ denote the random variable indicating the new state after a single jump.
Under $\P^n_{\eta,j}$, let $(Y,J)$ have distribution $T^n((\eta, j), \cdot)$, so that
\[ \P^n_{\eta,j}[\varphi(Y,J)] = \sum_{y,k} T^n((\eta, j),(y,k)) \varphi(y,k).\]

We will see that the correct speed-up factor for the LMH chain is $\alpha = (1-\gamma)$ (as opposed to $\alpha = 2(1-\gamma)$ for Metropolis-Hastings). At this scaling, the second moment of the increments vanishes for the LMH Markov chain:

\begin{lemma}[LMH second moment for Curie-Weiss]
\label{lem:higher-moments-LMH}
Let $0 < \gamma < 1$. Let $\alpha = 1-\gamma$. Suppose $h \neq 0$, $\beta \geq 0$ or $h = 0$, $0 \leq \beta \leq 1$.
Then for any $p > 1$ and $j \in \{-1, +1\}$,
\[ \lim_{n \rightarrow \infty} \sup_{\eta \in X^n}  n^{\alpha} \E_{\eta,j}^n[|Y-\eta|^p] = 0.\]
\end{lemma}

\begin{proof}
We compute
\begin{align*}
n^{\alpha} \E^n_{y,j=\pm}[|Y-\eta|^p] = n^{\alpha} p^n_{\pm}(\eta) (2 n^{\gamma -1})^p = 2^p p^n_{\pm}(\eta) n^{(p-1)(\gamma - 1)}
\end{align*}
Since $|p^n_{\pm}(\eta)| \leq 1$, the supremum over $\eta \in X^n$ converges to zero.
\end{proof}

Asymptotically, the first moment of the increments does not depend on $\eta$. Let $a(h,\beta)$ be given by~\eqref{eq:LMH-drift}.

\begin{lemma}[LMH drift for Curie-Weiss]
\label{lem:LMH-drift}
Let $0 < \delta < \gamma < 1$. Let $\alpha = 1 - \gamma$.  Suppose $h \neq 0$, $\beta \geq 0$ or $h = 0$, $0 \leq \beta \leq 1$.
Then for $j \in \{-1, +1\}$,
\[ \lim_{n \rightarrow \infty} \sup_{\eta \in X^n} \left| n^{\alpha} \E_{\eta,j}^n[Y-\eta] - a(h, \beta)  j \right|= 0.\]
\end{lemma}
\begin{proof}
We compute $n^{\alpha} \E_{\eta,\pm 1}^n[Y- \eta] = \pm p_{\pm}^n(\eta) 2n^{\gamma - 1} n^{\alpha} = \pm 2 p_{\pm}^n(\eta)$.
We may apply Lemma~\ref{lem:simplify} with $r = 0$ and $k = 0$, to replace $p_{\pm}(\eta)$ by $p^{n,0}_{\pm}(\eta)$, given by~\eqref{eq:simplified-transitions}.
Since $n^{-\gamma} |\eta| \leq n^{\delta - \gamma} \rightarrow 0$ in the supremum over $F^{n,\delta}$, we find that as $n \rightarrow \infty$, using~\eqref{eq:curie-weiss}, 
\[ \pm 2 p_{\pm}^{n,0}(\eta) \rightarrow \pm (1 \mp m_0) \min(1, \exp(\pm 2 \beta(m_0+h))) = \pm (1 \mp m_0)\min \left(1, \left(\frac{1+m_0}{1-m_0}\right)^{\pm 1} \right).\]
By distinguishing cases, this can be seen to equal~\eqref{eq:LMH-drift}.
\end{proof}

It only remains to determine the switching rates between the replicas. This will depend on whether $0 \leq \beta < 1$ or $\beta = 1$.

\subsubsection{Supercritical temperature regime}
As we have seen, for $0 \leq \beta < 1$ the correct scaling is given by $\gamma = \half$.
In this case, we have the following asymptotic result for the switching rate between replicas.

\begin{lemma}
\label{lem:LMH-switching-rate-high-temperature}
Let $0 \leq \beta < 1$ and $h \in \R$. Suppose $\gamma = \half$, $\alpha = 1- \gamma = \half$ and $\delta \in (0, 1/4)$. Then for $j = \pm 1$,
\[ \lim_{n \rightarrow \infty} \sup_{\eta \in F^{n,\delta}} \left|n^{\alpha} \P^n_{\eta,j}(J = -j) - \max(0, j l(h,\beta) \eta)  \right| = 0,\]
with $l(h,\beta)$ given by~\eqref{eq:drift-high-temperature}.
\end{lemma}
\begin{proof}
We have
\[ n^{\alpha} \P^n_{\eta,j=\pm 1}(J = -j) = n^{1/2} \max(0, p^n_{\mp}(\eta) - p^n_{\pm}(\eta)).\]
Applying Lemma~\ref{lem:simplify} with $r = \half$ and $k = 1$, we find that $p^n_{\pm}(\eta)$ may be approximated to sufficient accuracy by $p^{n,1}_{\pm}(\eta)$, given by~\eqref{eq:simplified-transitions}. We distinguish cases.
\begin{itemize}
 \item Suppose $h = 0$ (and hence $m_0 = 0$) and $j = +1$. Then
 \begin{align*}
&  p_{-}^{n,1}(\eta) - p_{+}^{n,1}(\eta) \\
& = \half \left\{(1 + n^{-1/2} \eta) \min(1, 1 - 2 \beta n^{-1/2} \eta) - (1 - n^{-1/2} \eta) \min(1, 1+2 \beta n^{-1/2} \eta) \right\} \\
 & = (1- \beta) n^{-1/2} \eta - \sign(\eta) \beta n^{-1} \eta^2.
 \end{align*}
Using Lipschitz continuity of $x \mapsto \max(0,x)$ and $\delta < 1/4$,
\[ |n^{1/2}\max(0, p^{n,1}_-(\eta) - p^{n,1}_+(\eta)) - \max(0, (1- \beta)  \eta)| \leq \beta n^{-1/2} \eta^2 \leq \beta n^{-1/2 + 2 \delta} \rightarrow 0\]
in the supremum over $\eta$, as $n \rightarrow \infty$. The case $j = -1$ is analogous.
\item Suppose $h \neq 0$. Let us say without loss of generality $h > 0$ and hence $m_0 > 0$. Taking $j = +1$, we compute using~\eqref{eq:curie-weiss},
\begin{align*}
&  p_-^{n,1}(\eta) - p_+^{n,1}(\eta) \\
& = \half(1 + m_0 + n^{-1/2}\eta)\exp(-2 \beta(m_0+h))(1 - 2 \beta n^{-1/2} \eta) - \half(1 - m_0 -n^{-1/2} \eta) \\
 & = \half(1 + m_0 + n^{-1/2}\eta)\left( \frac{1 - m_0}{1 + m_0} \right) (1 - 2 \beta n^{-1/2} \eta) - \half(1 - m_0 -n^{-1/2} \eta) \\
 & = \half \left\{n^{-1/2} \eta \left( \frac{1-m_0}{1 + m_0}\right) - 2 \beta n^{-1/2} \eta (1-m_0) + n^{-1/2}\eta \right\} + O(\eta^2 n^{-1}) \\
 & = \left(\frac{1}{1 + m_0} - \beta(1-m_0)\right)\eta.
\end{align*}
The other cases follow by analogous computations, or by exploiting the symmetry transformations $(\eta,j) \leftrightarrow (-\eta,-j)$ and $(h,m_0)\leftrightarrow(-h,-m_0)$.
\end{itemize}
\end{proof}

\begin{proof}[Proof of Theorem~\ref{thm:LMH-supercritical-temperature}]
The generator of $(Y^n, J^n)$ is given by 
\begin{equation}
 \label{eq:generator-LMH} L^{n,\alpha} \varphi(\eta,j) = n^{\alpha} (T^n \varphi(\eta, j) - \varphi(\eta, j)), \quad \varphi : X^n \times \{-1,+1\} \rightarrow \R.
\end{equation}
It is established in Proposition~\ref{prop:construction} that $L$ given by~\eqref{eq:generator-LMH}  generates a Markov process in $\R \times \{-1,1\}$. By Proposition~\ref{prop:feller} the Markov process corresponds to a Feller semigroup $(P(t))$ on $C_0(E)$.
Note 
\begin{align*} L^{n,\alpha} \varphi(\eta,j) & = n^{\alpha} \E_{\eta,j}^n [ \varphi(Y,J) - \varphi(y,j) ] \\
 & = n^{\alpha} \E_{\eta,j}^n [ (\varphi(Y,j) -\varphi(y,j))\1_{\{J = j\}}] +  (\varphi(y,-j) - \varphi(y,j))\P_{\eta,j}^n (J = -j)
\end{align*}
Consider the set of functions $M = \{ \varphi : \R \times \{-1,1\} \rightarrow \R, \mbox{$\varphi(\cdot,j) \in C^{\infty}_c(\R)$ for $j = \pm 1$}\}$. Then $M$ is strongly separating. Using an analogous Taylor approximation argument as in the proof of Theorem~\ref{thm:diffusion-limit-mh-high-temperature}, for $j = \pm 1$,
\begin{align*}
& \sup_{\eta \in F^{n,\delta}} \left|L^{n,\alpha} \varphi(\eta,j) - L \varphi(\eta,j)\right| \\
& \leq \sup_{\eta \in F^{n,\delta}} \left|L^{n,\alpha} \varphi(\eta, j) - \frac{\partial}{\partial \eta} \varphi(\eta,j) \E_{\eta,j}^n [Y - \eta] -n^{\alpha} \P_{\eta,j}^n(J = - j) (\varphi(\eta, -j) - \varphi(\eta, j)) \right| \\
& \quad + \sup_{\eta \in F^{n,\delta}} \left|\frac{\partial}{\partial \eta} \varphi(\eta,j) \E_{\eta,j}^n [Y - \eta]+ n^{\alpha} \P_{\eta,j}^n(J = - j) (\varphi(\eta, -j) - \varphi(\eta, j)) - L \varphi(\eta, j) \right| \\
& \leq \half n^{\alpha} \sup_{\eta \in F^{n, \delta}}\left\|\frac{\partial^2 \varphi}{\partial \eta^2}\right\|_{\infty}  \E_{\eta,j}^n[(Y - \eta)^2] + 2 \sup_{\eta \in F^{n,\delta}} |n^{\alpha} \P^n_{\eta,j}(J=-j) - \max(0, jl(h,\beta) \eta)| \| \varphi \|_{\infty}\\
& \quad + \sup_{\eta \in F^{n,\delta}} \left\| \frac{\partial \varphi}{\partial \eta} \right\|_{\infty} \left|  n^{\alpha} \E_{\eta, j}^n [Y - \eta] - a(h,\beta) j \right|
\end{align*}
which converges to zero by applying Lemmas~\ref{lem:higher-moments-LMH},~\ref{lem:LMH-drift} and~\ref{lem:LMH-switching-rate-high-temperature}, taking $\alpha = 1/2$, $\gamma = 1/2$ and $\delta  = 1/8$. As in the proof of Theorem~\ref{thm:diffusion-limit-mh-high-temperature}, using Lemma~\ref{lem:concentration-high-temperature} $(Y^n,J^n)$ are increasingly concentrated on $F^{n,\delta}$ for $\delta = 1/8$. We may now apply \cite[Corollary 4.8.7]{EthierKurtz2005} to deduce the stated weak convergence.
It is established in Proposition~\ref{prop:invariant-measure} that $(Y,J)$ has the stated stationary distribution.
\end{proof}

\subsubsection{At critical temperature}
As above for $h=0$ and $\beta = 1$ we consider the scaled magnetization $\eta^n = n^{\gamma} m^n$ with $\gamma = 1/4$ .

\begin{lemma}
\label{lem:LMH-switching-rate-critical-temperature}
Let $\beta = 1$, $h = 0$, $\gamma = 1/4$, $\alpha = 1-\gamma = 3/4$, and $\delta \in (0, 1/16)$.
Then 
\[ \lim_{n \rightarrow \infty} \sup_{\eta \in F^{n,\delta}} \left|n^{\alpha} \P^n_{\eta,j}(J = -j) - \max(0, 1/3 j \eta^3) \right| = 0.\]
\end{lemma}
\begin{proof}
As before
\[ n^{\alpha} \P^n_{\eta,j}(J = - j) = n^{3/4} \max(0, p^n_{\mp}(\eta) - p^n_{\pm}(\eta)).\]
Applying Lemma~\ref{lem:simplify} with $r = 3/4$, $\gamma = 1/4$ and $k = 3$, we find that a sufficiently precise approximation is $p^n_{\pm}(\eta) \approx p^{n,3}_{\pm}(\eta)$. 
The computation of $p^n_{\mp}(\eta) - p^n_{\pm}(\eta)$ has already been performed in Lemma~\ref{lem:drift-critical-temperature}, resulting in the stated expression.
\end{proof}

\begin{proof}[Proof of Theorem~\ref{thm:LMH-critical-temperature}]
The proof is fully analogous to the proof of Theorem~\ref{thm:LMH-supercritical-temperature}, now taking $\delta = 1/32$, $\gamma = 1/4$ and $\alpha = 3/4$, and applying Lemmas~\ref{lem:concentration-critical-temperature} and~\ref{lem:LMH-switching-rate-critical-temperature} instead of Lemmas~\ref{lem:concentration-high-temperature} and~\ref{lem:LMH-switching-rate-high-temperature}.
\end{proof}

\subsection{The limiting zig-zag process}
\label{sec:proof-zigzag}
By rescaling the time variable and $\lambda^{\pm}$ if necessary, we may assume $a = 1$ without loss of generality throughout the proofs below .

\subsubsection{Construction}
\begin{lemma}
\label{lem:finite-switching-times}
Suppose Assumption~\ref{ass:construction} holds. Then for every $(y,j) \in E$, $\lim_{t \rightarrow \infty} F(t;y,j) = 0$. In particular, for every $(y,j) \in E$, $1 - F(\cdot; y,j)$ is the distribution function of a positive random variable that is almost surely finite.
\end{lemma}
\begin{proof}
We fix $(y,j) \in E$. Suppose $T$ is distributed according to $1-F$. Since $F$ is continuous at 0, $\P(T = 0) = 0$. By Assumption~\ref{ass:construction}, there exist $t_0$ and $\lambda_{\min}$ such that $\lambda(y + j s, j) \geq \lambda_{\min}$ for $s \geq t_0$.
Then for $t \geq t_0$, 
\[ F(t; y,j) = F(t_0;y,j) \exp \left(-\int_{t_0}^t \lambda (y + j s, j) \ d s \right) \leq F(t_0;y,j) \exp\left( - (t-t_0) \lambda_{\min}\right),\]
and the stated result follows.
\end{proof}

\begin{lemma}
\label{lem:finitely-many-switches}
Suppose Assumption~\ref{ass:construction} holds. Then for every $t \geq 0$ and $(y,j) \in E$, $\P_{y,j}(N(t) < \infty) = 1$.
\end{lemma}

\begin{proof}
We assume $y,j$ are fixed and suppress the $(y,j)$-subscript in $\P_{y,j}$ etc. Introduce the notation $Y_k = Y(T_k)$, $J_k = J(T_k)$. 
Observe that on $\{T_{k-1} \leq t \leq T_k\}$,
\begin{align*} |Y(t) - y| & =  \left| J_{k-1} (t - T_{k-1}) + Y_{k-1} - y \right| \leq t - T_{k-1} +  \left| \sum_{i=1}^{k-1} (Y_i - Y_{i-1}) \right| \\
& \leq t- T_{k-1} + T_{k-1} = t.
\end{align*}
It follows that on $\{0  \leq t \leq T_k\}$, for every $s \leq t$,  $Y(s) \in [y- s, y+s] \subset [y- t, y+ t]$.
By Assumption~\ref{ass:construction}, $\lambda$ is bounded on $[y - t, y+ t]$, say by a constant $\lambda_{\max} > 0$.  

We will show by induction that, for $0 \leq s \leq t$ and $k \in \N \cup \{0\}$,
\[ \P(T_k \leq s) \leq 1 - \exp(-\lambda_{\max} s)  \sum_{j=0}^{k-1} \frac{(\lambda_{\max} s)^j}{j!}.
\]
For $k = 0$ this is trivial.
We have for any $k \in \N$ and $0 \leq s \leq t$,
\begin{align*} \P(T_k \leq s \mid T_{k-1}) & = \E \left[ \P(T_k \leq s \mid Y_{k-1},J_{k-1}) \1_{\{T_{k-1} \leq s\}}\mid T_{k-1}  \right] \\
 & = \E_{y,j} \left[ \left(1 - \exp \left( - \int_0^{s-T_{k-1}} \lambda(Y(r),J(r)) \ d r \right) \right) \1_{\{T_{k-1} \leq s\}} \mid T_{k-1} \right]  \\
 & \leq \left( 1 - \exp( -\lambda_{\max}(s-T_{k-1}) ) \right)\1_{\{T_{k-1} \leq s\}}.
\end{align*}
For $s \leq t$ it follows that
\begin{align}
\nonumber \P_{y,j}(T_k \leq s) & = \E_{y,j} \left[ \P_{y,j}(T_k \leq s  \mid T_{k-1}) \right] \leq  \E \left[  \left( 1 - \exp( -\lambda_{\max}(s-T_{k-1}) ) \right)\1_{\{T_{k-1} \leq s\}}  \right] \\
\label{eq:poisson-rate} & = \P(T_{k-1} \leq s) - \exp(-\lambda_{\max} s)\E \left[  \exp( \lambda_{\max} T_{k-1} )\1_{\{T_{k-1} \leq s\}}  \right].
\end{align}
Let $G$ denote the distribution function of $T_{k-1}$ and note by the induction hypothesis for $k-1$, 
\[ G(r) = \P(T_{k-1} \leq r) \leq 1 - \exp(-\lambda_{\max} r) \sum_{i=0}^{k-2} \frac{(\lambda_{\max} r)^i}{i!}, \quad 0 \leq r \leq t.\] 
Then
\begin{align*}
 & \E\left[ \exp( \lambda_{\max} T_{k-1} )\1_{\{T_{k-1} \leq s\}}  \right] \\
 & = \int_0^s \exp(\lambda_{\max} r) d G(r) = \left[ \exp(\lambda_{\max} r) G(r) \right]_0^s - \lambda_{\max} \int_0^s \exp(\lambda_{\max} r) G(r) \ d r \\
& \geq \exp(\lambda_{\max} s) \P(T_{k-1} \leq s) - \lambda_{\max} \int_0^s \exp(\lambda_{\max} r) \left( 1 -\exp(-\lambda_{\max} r) \sum_{i=0}^{k-2} \frac{(\lambda_{\max} r)^i}{i!} \right)\ d r  \\
& = \exp(\lambda_{\max} s) \P(T_{k-1} \leq s) + 1 - \exp(\lambda_{\max} s) + \sum_{i=0}^{k-2} \frac{(\lambda_{\max} s)^{i+1}}{(i+1)!}.
\end{align*}
Inserting this expression into~\eqref{eq:poisson-rate}, the induction hypothesis follows for $k$.
It now follows by the Fatou Lemma that
\[ P(N(t) = \infty) = \P(\liminf_{k \rightarrow \infty} \{T_k \leq t\}) \leq \liminf_{k \rightarrow \infty} \P(T_k \leq t) = 0.\]
\end{proof}

\subsubsection{Regularity}
\label{sec:proofs-regularity}

The total variation distance between measures on a Polish space is defined as usual by
\[ \| \nu - \mu\|_{\mathrm{TV}} := \sup_{A} | \nu(A) - \mu(A)|,\]
where the supremum is over all Borel sets.

\begin{proof}[Proof of Proposition~\ref{prop:feller}]
Let $\varphi \in C_0(E)$. The value of $P(t) \varphi(y,j)$ only depends on values of $\varphi$ within the bounded set $([y - t, y + t], \pm 1) \subset E$. Since $\varphi$ vanishes at infinity $P(t) \varphi$ vanishes at infinity as well. It remains to establish continuity of $P(t) \varphi$.

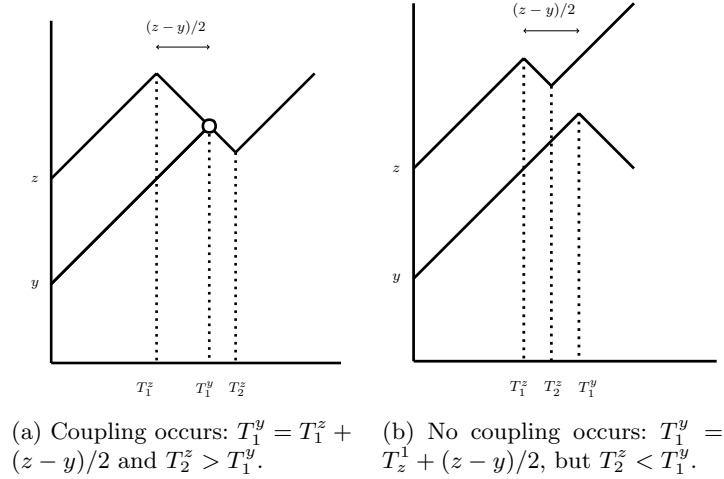
\begin{figure}[ht]
\begin{subfigure}[b]{0.3 \textwidth}\resizebox{\textwidth}{!}{
\ifx\du\undefined
  \newlength{\du}
\fi
\setlength{\du}{15\unitlength}
\begin{tikzpicture}
\pgftransformxscale{1.000000}
\pgftransformyscale{-1.000000}
\definecolor{dialinecolor}{rgb}{0.000000, 0.000000, 0.000000}
\pgfsetstrokecolor{dialinecolor}
\definecolor{dialinecolor}{rgb}{1.000000, 1.000000, 1.000000}
\pgfsetfillcolor{dialinecolor}
\pgfsetlinewidth{0.100000\du}
\pgfsetdash{}{0pt}
\pgfsetdash{}{0pt}
\pgfsetbuttcap
{
\definecolor{dialinecolor}{rgb}{0.000000, 0.000000, 0.000000}
\pgfsetfillcolor{dialinecolor}
\definecolor{dialinecolor}{rgb}{0.000000, 0.000000, 0.000000}
\pgfsetstrokecolor{dialinecolor}
\draw (5.000000\du,2.000000\du)--(5.000000\du,15.000000\du);
}
\pgfsetlinewidth{0.100000\du}
\pgfsetdash{}{0pt}
\pgfsetdash{}{0pt}
\pgfsetbuttcap
{
\definecolor{dialinecolor}{rgb}{0.000000, 0.000000, 0.000000}
\pgfsetfillcolor{dialinecolor}
\definecolor{dialinecolor}{rgb}{0.000000, 0.000000, 0.000000}
\pgfsetstrokecolor{dialinecolor}
\draw (16.000000\du,15.000000\du)--(5.000000\du,15.000000\du);
}
\pgfsetlinewidth{0.100000\du}
\pgfsetdash{}{0pt}
\pgfsetdash{}{0pt}
\pgfsetbuttcap
{
\definecolor{dialinecolor}{rgb}{0.000000, 0.000000, 0.000000}
\pgfsetfillcolor{dialinecolor}
\definecolor{dialinecolor}{rgb}{0.000000, 0.000000, 0.000000}
\pgfsetstrokecolor{dialinecolor}
\draw (5.000000\du,8.000000\du)--(9.000000\du,4.000000\du);
}
\pgfsetlinewidth{0.100000\du}
\pgfsetdash{}{0pt}
\pgfsetdash{}{0pt}
\pgfsetbuttcap
{
\definecolor{dialinecolor}{rgb}{0.000000, 0.000000, 0.000000}
\pgfsetfillcolor{dialinecolor}
\definecolor{dialinecolor}{rgb}{0.000000, 0.000000, 0.000000}
\pgfsetstrokecolor{dialinecolor}
\draw (9.000000\du,4.000000\du)--(12.000000\du,7.000000\du);
}
\pgfsetlinewidth{0.100000\du}
\pgfsetdash{{\pgflinewidth}{0.200000\du}}{0cm}
\pgfsetdash{{\pgflinewidth}{0.200000\du}}{0cm}
\pgfsetbuttcap
{
\definecolor{dialinecolor}{rgb}{0.000000, 0.000000, 0.000000}
\pgfsetfillcolor{dialinecolor}
\definecolor{dialinecolor}{rgb}{0.000000, 0.000000, 0.000000}
\pgfsetstrokecolor{dialinecolor}
\draw (9.000000\du,4.000000\du)--(9.000000\du,15.000000\du);
}
\pgfsetlinewidth{0.100000\du}
\pgfsetdash{{\pgflinewidth}{0.200000\du}}{0cm}
\pgfsetdash{{\pgflinewidth}{0.200000\du}}{0cm}
\pgfsetbuttcap
{
\definecolor{dialinecolor}{rgb}{0.000000, 0.000000, 0.000000}
\pgfsetfillcolor{dialinecolor}
\definecolor{dialinecolor}{rgb}{0.000000, 0.000000, 0.000000}
\pgfsetstrokecolor{dialinecolor}
\draw (11.000000\du,6.000000\du)--(11.000000\du,15.000000\du);
}
\definecolor{dialinecolor}{rgb}{0.000000, 0.000000, 0.000000}
\pgfsetstrokecolor{dialinecolor}
\node[anchor=west] at (4.000000\du,12.000000\du){$y$};
\definecolor{dialinecolor}{rgb}{0.000000, 0.000000, 0.000000}
\pgfsetstrokecolor{dialinecolor}
\node[anchor=west] at (4.000000\du,8.000000\du){$z$};
\definecolor{dialinecolor}{rgb}{0.000000, 0.000000, 0.000000}
\pgfsetstrokecolor{dialinecolor}
\node[anchor=west] at (8.000000\du,16.000000\du){$T_1^z$};
\definecolor{dialinecolor}{rgb}{0.000000, 0.000000, 0.000000}
\pgfsetstrokecolor{dialinecolor}
\node[anchor=west] at (10.2500000\du,16.000000\du){$T_1^y$};
\pgfsetlinewidth{0.100000\du}
\pgfsetdash{}{0pt}
\pgfsetdash{}{0pt}
\pgfsetbuttcap
{
\definecolor{dialinecolor}{rgb}{0.000000, 0.000000, 0.000000}
\pgfsetfillcolor{dialinecolor}
\definecolor{dialinecolor}{rgb}{0.000000, 0.000000, 0.000000}
\pgfsetstrokecolor{dialinecolor}
\draw (5.000000\du,12.000000\du)--(11.212132\du,5.787868\du);
}
\definecolor{dialinecolor}{rgb}{0.000000, 0.000000, 0.000000}
\pgfsetstrokecolor{dialinecolor}
\draw (5.000000\du,12.000000\du)--(10.823223\du,6.176777\du);
\pgfsetlinewidth{0.100000\du}
\pgfsetdash{}{0pt}
\pgfsetmiterjoin
\pgfsetbuttcap
\definecolor{dialinecolor}{rgb}{1.000000, 1.000000, 1.000000}
\pgfsetfillcolor{dialinecolor}
\pgfpathmoveto{\pgfpoint{11.176777\du}{5.823223\du}}
\pgfpathcurveto{\pgfpoint{11.265165\du}{5.911612\du}}{\pgfpoint{11.265165\du}{6.088388\du}}{\pgfpoint{11.176777\du}{6.176777\du}}
\pgfpathcurveto{\pgfpoint{11.088388\du}{6.265165\du}}{\pgfpoint{10.911612\du}{6.265165\du}}{\pgfpoint{10.823223\du}{6.176777\du}}
\pgfpathcurveto{\pgfpoint{10.734835\du}{6.088388\du}}{\pgfpoint{10.734835\du}{5.911612\du}}{\pgfpoint{10.823223\du}{5.823223\du}}
\pgfpathcurveto{\pgfpoint{10.911612\du}{5.734835\du}}{\pgfpoint{11.088388\du}{5.734835\du}}{\pgfpoint{11.176777\du}{5.823223\du}}
\pgfusepath{fill}
\definecolor{dialinecolor}{rgb}{0.000000, 0.000000, 0.000000}
\pgfsetstrokecolor{dialinecolor}
\pgfpathmoveto{\pgfpoint{11.176777\du}{5.823223\du}}
\pgfpathcurveto{\pgfpoint{11.265165\du}{5.911612\du}}{\pgfpoint{11.265165\du}{6.088388\du}}{\pgfpoint{11.176777\du}{6.176777\du}}
\pgfpathcurveto{\pgfpoint{11.088388\du}{6.265165\du}}{\pgfpoint{10.911612\du}{6.265165\du}}{\pgfpoint{10.823223\du}{6.176777\du}}
\pgfpathcurveto{\pgfpoint{10.734835\du}{6.088388\du}}{\pgfpoint{10.734835\du}{5.911612\du}}{\pgfpoint{10.823223\du}{5.823223\du}}
\pgfpathcurveto{\pgfpoint{10.911612\du}{5.734835\du}}{\pgfpoint{11.088388\du}{5.734835\du}}{\pgfpoint{11.176777\du}{5.823223\du}}
\pgfusepath{stroke}
\pgfsetlinewidth{0.000000\du}
\pgfsetdash{}{0pt}
\pgfsetdash{}{0pt}
\pgfsetbuttcap
{
\definecolor{dialinecolor}{rgb}{0.000000, 0.000000, 0.000000}
\pgfsetfillcolor{dialinecolor}
\pgfsetarrowsstart{to}
\pgfsetarrowsend{to}
\definecolor{dialinecolor}{rgb}{0.000000, 0.000000, 0.000000}
\pgfsetstrokecolor{dialinecolor}
\draw (9.000000\du,3.000000\du)--(11.000000\du,3.000000\du);
}
\definecolor{dialinecolor}{rgb}{0.000000, 0.000000, 0.000000}
\pgfsetstrokecolor{dialinecolor}
\node[anchor=west] at (8.400000\du,2.250000\du){$(z-y)/2$};
\pgfsetlinewidth{0.100000\du}
\pgfsetdash{}{0pt}
\pgfsetdash{}{0pt}
\pgfsetbuttcap
{
\definecolor{dialinecolor}{rgb}{0.000000, 0.000000, 0.000000}
\pgfsetfillcolor{dialinecolor}
\definecolor{dialinecolor}{rgb}{0.000000, 0.000000, 0.000000}
\pgfsetstrokecolor{dialinecolor}
\draw (12.000000\du,7.000000\du)--(15.000000\du,4.000000\du);
}
\definecolor{dialinecolor}{rgb}{0.000000, 0.000000, 0.000000}
\pgfsetstrokecolor{dialinecolor}
\node[anchor=west] at (11.500000\du,16.000000\du){$T_2^z$};
\pgfsetlinewidth{0.100000\du}
\pgfsetdash{{\pgflinewidth}{0.200000\du}}{0cm}
\pgfsetdash{{\pgflinewidth}{0.200000\du}}{0cm}
\pgfsetbuttcap
{
\definecolor{dialinecolor}{rgb}{0.000000, 0.000000, 0.000000}
\pgfsetfillcolor{dialinecolor}
\definecolor{dialinecolor}{rgb}{0.000000, 0.000000, 0.000000}
\pgfsetstrokecolor{dialinecolor}
\draw (12.000000\du,7.000000\du)--(12.000000\du,15.000000\du);
}
\end{tikzpicture}
}
\caption{Coupling occurs: $T_1^y = T_1^z + (z-y)/2$ and $T_2^z > T_1^y$.}
\label{fig:coupling}
\end{subfigure}
\quad 
\begin{subfigure}[b]{0.3 \textwidth}\resizebox{\textwidth}{!}{
\ifx\du\undefined
  \newlength{\du}
\fi
\setlength{\du}{15\unitlength}
\begin{tikzpicture}
\pgftransformxscale{1.000000}
\pgftransformyscale{-1.000000}
\definecolor{dialinecolor}{rgb}{0.000000, 0.000000, 0.000000}
\pgfsetstrokecolor{dialinecolor}
\definecolor{dialinecolor}{rgb}{1.000000, 1.000000, 1.000000}
\pgfsetfillcolor{dialinecolor}
\pgfsetlinewidth{0.100000\du}
\pgfsetdash{}{0pt}
\pgfsetdash{}{0pt}
\pgfsetbuttcap
{
\definecolor{dialinecolor}{rgb}{0.000000, 0.000000, 0.000000}
\pgfsetfillcolor{dialinecolor}
\definecolor{dialinecolor}{rgb}{0.000000, 0.000000, 0.000000}
\pgfsetstrokecolor{dialinecolor}
\draw (5.000000\du,2.000000\du)--(5.000000\du,15.000000\du);
}
\pgfsetlinewidth{0.100000\du}
\pgfsetdash{}{0pt}
\pgfsetdash{}{0pt}
\pgfsetbuttcap
{
\definecolor{dialinecolor}{rgb}{0.000000, 0.000000, 0.000000}
\pgfsetfillcolor{dialinecolor}
\definecolor{dialinecolor}{rgb}{0.000000, 0.000000, 0.000000}
\pgfsetstrokecolor{dialinecolor}
\draw (16.000000\du,15.000000\du)--(5.000000\du,15.000000\du);
}
\pgfsetlinewidth{0.100000\du}
\pgfsetdash{}{0pt}
\pgfsetdash{}{0pt}
\pgfsetbuttcap
{
\definecolor{dialinecolor}{rgb}{0.000000, 0.000000, 0.000000}
\pgfsetfillcolor{dialinecolor}
\definecolor{dialinecolor}{rgb}{0.000000, 0.000000, 0.000000}
\pgfsetstrokecolor{dialinecolor}
\draw (5.000000\du,8.000000\du)--(9.000000\du,4.000000\du);
}
\pgfsetlinewidth{0.100000\du}
\pgfsetdash{}{0pt}
\pgfsetdash{}{0pt}
\pgfsetbuttcap
{
\definecolor{dialinecolor}{rgb}{0.000000, 0.000000, 0.000000}
\pgfsetfillcolor{dialinecolor}
\definecolor{dialinecolor}{rgb}{0.000000, 0.000000, 0.000000}
\pgfsetstrokecolor{dialinecolor}
\draw (9.000000\du,4.000000\du)--(10.000000\du,5.000000\du);
}
\pgfsetlinewidth{0.100000\du}
\pgfsetdash{{\pgflinewidth}{0.200000\du}}{0cm}
\pgfsetdash{{\pgflinewidth}{0.200000\du}}{0cm}
\pgfsetbuttcap
{
\definecolor{dialinecolor}{rgb}{0.000000, 0.000000, 0.000000}
\pgfsetfillcolor{dialinecolor}
\definecolor{dialinecolor}{rgb}{0.000000, 0.000000, 0.000000}
\pgfsetstrokecolor{dialinecolor}
\draw (9.000000\du,4.000000\du)--(9.000000\du,15.000000\du);
}
\pgfsetlinewidth{0.100000\du}
\pgfsetdash{{\pgflinewidth}{0.200000\du}}{0cm}
\pgfsetdash{{\pgflinewidth}{0.200000\du}}{0cm}
\pgfsetbuttcap
{
\definecolor{dialinecolor}{rgb}{0.000000, 0.000000, 0.000000}
\pgfsetfillcolor{dialinecolor}
\definecolor{dialinecolor}{rgb}{0.000000, 0.000000, 0.000000}
\pgfsetstrokecolor{dialinecolor}
\draw (11.000000\du,6.000000\du)--(11.000000\du,15.000000\du);
}
\definecolor{dialinecolor}{rgb}{0.000000, 0.000000, 0.000000}
\pgfsetstrokecolor{dialinecolor}
\node[anchor=west] at (4.000000\du,12.000000\du){$y$};
\definecolor{dialinecolor}{rgb}{0.000000, 0.000000, 0.000000}
\pgfsetstrokecolor{dialinecolor}
\node[anchor=west] at (4.000000\du,8.000000\du){$z$};
\definecolor{dialinecolor}{rgb}{0.000000, 0.000000, 0.000000}
\pgfsetstrokecolor{dialinecolor}
\node[anchor=west] at (8.250000\du,16.000000\du){$T_1^z$};
\definecolor{dialinecolor}{rgb}{0.000000, 0.000000, 0.000000}
\pgfsetstrokecolor{dialinecolor}
\node[anchor=west] at (10.750000\du,16.000000\du){$T_1^y$};
\pgfsetlinewidth{0.100000\du}
\pgfsetdash{}{0pt}
\pgfsetdash{}{0pt}
\pgfsetbuttcap
{
\definecolor{dialinecolor}{rgb}{0.000000, 0.000000, 0.000000}
\pgfsetfillcolor{dialinecolor}
\definecolor{dialinecolor}{rgb}{0.000000, 0.000000, 0.000000}
\pgfsetstrokecolor{dialinecolor}
\draw (5.000000\du,12.000000\du)--(11.000000\du,6.000000\du);
}
\pgfsetlinewidth{0.000000\du}
\pgfsetdash{}{0pt}
\pgfsetdash{}{0pt}
\pgfsetbuttcap
{
\definecolor{dialinecolor}{rgb}{0.000000, 0.000000, 0.000000}
\pgfsetfillcolor{dialinecolor}
\pgfsetarrowsstart{to}
\pgfsetarrowsend{to}
\definecolor{dialinecolor}{rgb}{0.000000, 0.000000, 0.000000}
\pgfsetstrokecolor{dialinecolor}
\draw (9.000000\du,3.000000\du)--(11.000000\du,3.000000\du);
}
\definecolor{dialinecolor}{rgb}{0.000000, 0.000000, 0.000000}
\pgfsetstrokecolor{dialinecolor}
\node[anchor=west] at (8.400000\du,2.250000\du){$(z-y)/2$};
\pgfsetlinewidth{0.100000\du}
\pgfsetdash{}{0pt}
\pgfsetdash{}{0pt}
\pgfsetbuttcap
{
\definecolor{dialinecolor}{rgb}{0.000000, 0.000000, 0.000000}
\pgfsetfillcolor{dialinecolor}
\definecolor{dialinecolor}{rgb}{0.000000, 0.000000, 0.000000}
\pgfsetstrokecolor{dialinecolor}
\draw (10.000000\du,5.000000\du)--(13.000000\du,2.000000\du);
}
\pgfsetlinewidth{0.100000\du}
\pgfsetdash{}{0pt}
\pgfsetdash{}{0pt}
\pgfsetbuttcap
{
\definecolor{dialinecolor}{rgb}{0.000000, 0.000000, 0.000000}
\pgfsetfillcolor{dialinecolor}
\definecolor{dialinecolor}{rgb}{0.000000, 0.000000, 0.000000}
\pgfsetstrokecolor{dialinecolor}
\draw (11.000000\du,6.000000\du)--(13.000000\du,8.000000\du);
}
\pgfsetlinewidth{0.100000\du}
\pgfsetdash{{\pgflinewidth}{0.200000\du}}{0cm}
\pgfsetdash{{\pgflinewidth}{0.200000\du}}{0cm}
\pgfsetbuttcap
{
\definecolor{dialinecolor}{rgb}{0.000000, 0.000000, 0.000000}
\pgfsetfillcolor{dialinecolor}
\definecolor{dialinecolor}{rgb}{0.000000, 0.000000, 0.000000}
\pgfsetstrokecolor{dialinecolor}
\draw (10.000000\du,5.000000\du)--(10.000000\du,15.000000\du);
}
\definecolor{dialinecolor}{rgb}{0.000000, 0.000000, 0.000000}
\pgfsetstrokecolor{dialinecolor}
\node[anchor=west] at (9.500000\du,16.000000\du){$T_2^z$};
\end{tikzpicture}}
\caption{No coupling occurs: $T_1^y = T^1_z + (z-y)/2$, but $T_2^z < T_1^y$.}
\label{fig:no_coupling}
\end{subfigure}
\caption{Illustration of the coupling used in the proof of Proposition~\ref{prop:feller}.}
\label{fig:coupling-no-coupling}
\end{figure}

We construct a coupling of $(Y,J)$ starting from two different initial conditions, $(y,j)$ and $(z,j)$, as follows. 
Let $(y,j), (z,j) \in E$ and suppose $z \in \R$. Without loss of generality assume $j = +1$ and $z \geq y$. Let $\nu_1$ denote the distribution of $T_1 + (z-y)/2$, with initial condition $(z,j)$, i.e. $\nu_1$ has distribution function
\[ H_1(t;y,z) = \P_{z,j} (T_1 + (z-y)/2 \leq t) = \P_{z,j}(T_1 \leq t - (z-y)/2) = 1 - F(t-(z-y)/2; z,j),\]
and let $\nu_2$ denote the distribution of $T_1$ with initial condition $(y,j)$, i.e. $\nu_2$ has distribution function
\[ H_2(t;y,z) = \P_{y,j} (T_1 \leq t) = 1 - F(t;y,j).\]
Let $c_1(y,z) := \| \nu_1 - \nu_2 \|_{\mathrm{TV}}$. There exists a `maximal' coupling $(R_1, R_2)$ under a probability measure $\P$ of $\nu_1$ and $\nu_2$ such that $\P(R_1 \neq R_2) = c_1(y,z)$, see e.g. \cite[Theorem I.5.2]{Lindvall2002}.
Use $(Y^y,J^y)$ to denote the process starting from initial condition $(y,j)$ and $(Y^z,J^z)$ for the process starting from initial condition $(z,j)$. We introduce a dependence between the two processes through the distribution of the first replica switch time, $T_1$. Using the same superscript notation here, we let $T_1^z = R_1 - (z-y)/2$ and $T_1^y := R_2$.
Let all other switch times $T_i^y$ and $T_i^z$ be defined as usual, i.e. $T_{i+1}^y = T_i^y + Z_{i+1}^y$ where $\P(Z_{i+1} \geq \zeta) = F(\zeta; Y_{T_i^y}^y,J_{T_i^y}^y)$, etc. and construct the continuous time processes $(Y^y, J^y)$ and $(Y^z,J^z)$ as in Section~\ref{sec:construction}.
Define an event
\[ \Omega_{\mathrm{coupling}} := \{ R_1=R_2 \ \mbox{and} \ T_2^z > T_1^y\},
\]
i.e. on $\Omega_{\mathrm{coupling}}$ a coupling occurs between $R_1$ and $R_2$, and $(Y^z,J^z)$ does not switch a second time before $T_1^y$. On $\Omega_{\mathrm{coupling}}$, 
\[ T_1^y = R_2 = R_1 = T_1^z + (z-y)/2.\] 
and hence
\[ Y^z(T_1^y) = z + T_1^z - (T_1^y - T_1^z) = z + T_1^z - (z-y)/2 = (y+z)/2 + T_1^z,\]
and 
\[ Y^y(T_1^y) = y + T_1^y = y + T_1^z + (z-y)/2 = (y+z)/2 + T_1^z,\]
i.e. $Y^z(T_1^y) = Y^y(T_1^y)$. By the Strong Markov property, the process
\[ (\widetilde Y^y, \widetilde J^y)(t, \omega) := \left\{ \begin{array}{ll} (Y^z(t, \omega), J^z(t, \omega)) \quad & \mbox{$\omega \in \Omega_{\mathrm{coupling}}$, $t \geq T_1^y$}, \\
                                                           (Y^y(t, \omega), J^y(t, \omega)) \quad & \mbox{otherwise,}
                                                          \end{array} \right.
\]
is a Markov process with generator $L$.
Since $H_1$ and $H_2$ have densities, we may evaluate
\begin{align*}
& c_1(y,z)  = \| \nu_1 - \nu_2 \|_{\mathrm{TV}} \\
& = \half \int_0^{\infty} |H_1'(t) - H_2'(t)| \ d t = \half \int_0^{\frac{z-y}{2}} |H_2'(t)| \ d t + \half \int_{\frac{z-y}{2}}^{\infty} |H_1'(t) - H_2'(t)| \ d t  \\
 & = \half \int_0^{\frac{z-y}{2}} \lambda(y + t) F(t;y,j) \ d t \\
 & \quad + \half \int_{\frac{z-y}{2}} \left| \lambda(z + t-(z-y)/2) F(t - (z-y)/2; z, j) - \lambda(y + t) F(t;y,j)\right|  \ d t
\end{align*}
The second integrand is trivially dominated by
\[  \lambda(z + t-(z-y)/2) F(t - (z-y)/2; z, j) + \lambda(y + t) F(t;y,j),\] which is integrable (since it is the sum of two density functions). Since $\lambda$ and $F$ depend continuously on $y,z$, we may apply the dominated convergence theorem to conclude that $c_1(y,z)$ is continuous in $y,z$. Also note that $c_1(y,y) = c_1(z,z) = 0$. Hence $\lim_{y \rightarrow z} c_1(y,z) = \lim_{z \rightarrow y} c_1(y,z)= 0$.
Also let
\begin{align*}
 c_2(y,z) = \P_{z,j}(T_2^z \leq T_1^z + (z-y)/2) = 1 - \E_{z,j}[ F((z-y)/2; Y(T_1),J(T_1))],
\end{align*}
and note that $c_2$ is continuous in $(y,z)$ and $\lim_{y \rightarrow z} c_2(y,z) = 0$. We estimate
\begin{align*} \P(\Omega \setminus \Omega_{\mathrm{coupling}}) & = \P(R_1 \neq R_2 \ \mbox{or} \ T_2^z \leq T_1^y) = \P(R_1 \neq R_2 \ \mbox{or} \ T_2^z \leq T_1^z + (z-y)/2 )\\
 & \leq \P(R_1 \neq R_2) + \P(T_2^z \leq T_1^z + (z-y)/2) \\
 & = c_1(y,z) + c_2(y,z).
\end{align*}

Fix $t \geq 0$. Let $\varepsilon > 0$ and let $y \in \R$. Pick $\delta > 0$ such that $c_1(y,z) + c_2(y,z) < \varepsilon / (2 \|\varphi\|_{\infty})$ for all $z$ for which  $|y-z|<\delta$ and (using uniform continuity) $|\varphi(\zeta_1) - \varphi(\zeta_2)| < \varepsilon / 2$ for all $\zeta_1,\zeta_2: |\zeta_1 - \zeta_2| < \delta$ with $\zeta_i \in [y - t, z+t]$.
Then, for $|y-z|<\delta$, using that on $\Omega_{\mathrm{coupling}}$, the processes $Y^y(t)$ and $Y^z(t)$ remain within distance $|y-z|$ of each other and within $[y-t,z+t]$, we estimate
\begin{align*} |\E_{y,j} \varphi(Y(t),J(t)) - \E_{z,j} \varphi(Y(t),J(t))| & \leq \E |\varphi(Y^y(t),J^y(t)) - \varphi(Y^z(t), J^z(t))|  \\
 & \leq \P(\Omega_{\mathrm{coupling}}) \varepsilon/2 +  (c_1(y,z) + c_2(y,z)) \| \varphi \|_{\infty} \\
& < \varepsilon,
\end{align*}
which establishes continuity of $P(t)\varphi(y,j)$ in $y$ for $j=+1$. The case $j = -1$ is analogous.
\end{proof}

\subsubsection{Petite sets}

Let $K$ denote the resolvent Markov kernel given by
\[ K((y,j),A) = \int_0^{\infty} \exp(-t) \P_{y,j}((Y(t),J(t)) \in A)  \ d t, \quad (y,j) \in E, A \in \mathcal B(E).\]
The notion of a petite set plays an important role in establishing exponential ergodicity for a continuous time Markov process, see e.g. \cite{MeynTweedie1993-III}. A set $C \subset E$ is \emph{petite} for $K$ if there exists a nontrivial reference measure $\nu$ on $E$ such that $K((y,j),A) \geq \nu(A)$, for any $(y,j) \in C$ and $A \in \mathcal B(E)$.
The following lemma is instrumental in establishing exponential ergodicity (Theorem~\ref{thm:exponential-ergodicity}).

\begin{lemma}
\label{lem:compact-sets-are-petite}
Suppose Assumption~\ref{ass:construction} holds. Then every compact set $C \subset E$ is petite for $K$.
\end{lemma}

\begin{proof}
Let $y_0 \geq 0$ and $\lambda_{\min} > 0$ be as defined in Assumption~\ref{ass:construction}(ii).
Without loss of generality, it is sufficient to show that any set $C$ of the form $C := [-y_1, y_1] \times \{-1,1\}$, with $y_1 \geq y_0$, is petite.  Indeed, given a compact set $\widetilde C$ choose $y_1 \geq y_0$ sufficiently large such that $\widetilde C \subset C$. If $C$ is petite then clearly $\widetilde C$ is petite.

Let $C = [-1,1] \times \{-1,1\}$ and $\delta > 0$. We will show that for any $\delta > 0$ there exists a constant $c > 0$ such that for every $(y,j) \in C$ and $\varphi \in B_b(E)$, $\varphi \geq 0$,
\begin{equation} \label{eq:petite}\int_0^{\infty} \exp(-t) \E_{y,j} [\varphi(Y(t),J(t))] \ d t \geq c \int_{y_1}^{(1+ \delta)y_1} \varphi(z,+1) \ d z.\end{equation}
This then establishes that $C$ is $\nu$-petite with $\nu$ proportional to Lebesgue measure on $[y_1,(1+\delta)y_1]\times\{+1\}$.

\begin{figure}[ht] 
\begin{subfigure}[b]{0.48 \textwidth}
\resizebox{0.7 \textwidth}{!}{ 
\ifx\du\undefined
  \newlength{\du}
\fi
\setlength{\du}{15\unitlength}
\begin{tikzpicture}
\pgftransformxscale{1.000000}
\pgftransformyscale{-1.000000}
\definecolor{dialinecolor}{rgb}{0.000000, 0.000000, 0.000000}
\pgfsetstrokecolor{dialinecolor}
\definecolor{dialinecolor}{rgb}{1.000000, 1.000000, 1.000000}
\pgfsetfillcolor{dialinecolor}
\pgfsetlinewidth{0.100000\du}
\pgfsetdash{}{0pt}
\pgfsetdash{}{0pt}
\pgfsetbuttcap
{
\definecolor{dialinecolor}{rgb}{0.000000, 0.000000, 0.000000}
\pgfsetfillcolor{dialinecolor}
\definecolor{dialinecolor}{rgb}{0.000000, 0.000000, 0.000000}
\pgfsetstrokecolor{dialinecolor}
\draw (6.000000\du,11.000000\du)--(6.000000\du,26.000000\du);
}
\pgfsetlinewidth{0.100000\du}
\pgfsetdash{}{0pt}
\pgfsetdash{}{0pt}
\pgfsetbuttcap
{
\definecolor{dialinecolor}{rgb}{0.000000, 0.000000, 0.000000}
\pgfsetfillcolor{dialinecolor}
\definecolor{dialinecolor}{rgb}{0.000000, 0.000000, 0.000000}
\pgfsetstrokecolor{dialinecolor}
\draw (6.000000\du,20.000000\du)--(20.000000\du,20.000000\du);
}
\definecolor{dialinecolor}{rgb}{0.000000, 0.000000, 0.000000}
\pgfsetstrokecolor{dialinecolor}
\node[anchor=west] at (5.00000\du,16.000000\du){$y_1$};
\definecolor{dialinecolor}{rgb}{0.000000, 0.000000, 0.000000}
\pgfsetstrokecolor{dialinecolor}
\node[anchor=west] at (4.400000\du,24.000000\du){$-y_1$};
\definecolor{dialinecolor}{rgb}{0.000000, 0.000000, 0.000000}
\pgfsetstrokecolor{dialinecolor}
\node[anchor=west] at (3.000000\du,13.000000\du){$(1+\delta) y_1$};
\pgfsetlinewidth{0.000000\du}
\pgfsetdash{}{0pt}
\pgfsetdash{}{0pt}
\pgfsetmiterjoin
\definecolor{dialinecolor}{rgb}{0.678431, 0.847059, 0.901961}
\pgfsetfillcolor{dialinecolor}
\fill (6.000000\du,13.000000\du)--(6.000000\du,16.000000\du)--(20.000000\du,16.000000\du)--(20.000000\du,13.000000\du)--cycle;
\definecolor{dialinecolor}{rgb}{0.000000, 0.000000, 0.000000}
\pgfsetstrokecolor{dialinecolor}
\draw (6.000000\du,13.000000\du)--(6.000000\du,16.000000\du)--(20.000000\du,16.000000\du)--(20.000000\du,13.000000\du)--cycle;
\pgfsetlinewidth{0.100000\du}
\pgfsetdash{}{0pt}
\pgfsetdash{}{0pt}
\pgfsetbuttcap
{
\definecolor{dialinecolor}{rgb}{0.000000, 0.000000, 0.000000}
\pgfsetfillcolor{dialinecolor}
\definecolor{dialinecolor}{rgb}{0.000000, 0.000000, 0.000000}
\pgfsetstrokecolor{dialinecolor}
\draw (6.000000\du,24.000000\du)--(17.000000\du,13.000000\du);
}
\pgfsetlinewidth{0.100000\du}
\pgfsetdash{}{0pt}
\pgfsetdash{}{0pt}
\pgfsetbuttcap
{
\definecolor{dialinecolor}{rgb}{0.000000, 0.000000, 0.000000}
\pgfsetfillcolor{dialinecolor}
\definecolor{dialinecolor}{rgb}{0.000000, 0.000000, 0.000000}
\pgfsetstrokecolor{dialinecolor}
\draw (6.000000\du,18.000000\du)--(11.000000\du,13.000000\du);
}
\pgfsetlinewidth{0.000000\du}
\pgfsetdash{{\pgflinewidth}{0.200000\du}}{0cm}
\pgfsetdash{{\pgflinewidth}{0.200000\du}}{0cm}
\pgfsetbuttcap
{
\definecolor{dialinecolor}{rgb}{0.000000, 0.000000, 0.000000}
\pgfsetfillcolor{dialinecolor}
\definecolor{dialinecolor}{rgb}{0.000000, 0.000000, 0.000000}
\pgfsetstrokecolor{dialinecolor}
\draw (8.000000\du,16.000000\du)--(8.000000\du,20.000000\du);
}
\pgfsetlinewidth{0.000000\du}
\pgfsetdash{{\pgflinewidth}{0.200000\du}}{0cm}
\pgfsetdash{{\pgflinewidth}{0.200000\du}}{0cm}
\pgfsetbuttcap
{
\definecolor{dialinecolor}{rgb}{0.000000, 0.000000, 0.000000}
\pgfsetfillcolor{dialinecolor}
\definecolor{dialinecolor}{rgb}{0.000000, 0.000000, 0.000000}
\pgfsetstrokecolor{dialinecolor}
\draw (11.000000\du,20.000000\du)--(11.000000\du,13.000000\du);
}
\pgfsetlinewidth{0.000000\du}
\pgfsetdash{{\pgflinewidth}{0.200000\du}}{0cm}
\pgfsetdash{{\pgflinewidth}{0.200000\du}}{0cm}
\pgfsetbuttcap
{
\definecolor{dialinecolor}{rgb}{0.000000, 0.000000, 0.000000}
\pgfsetfillcolor{dialinecolor}
\definecolor{dialinecolor}{rgb}{0.000000, 0.000000, 0.000000}
\pgfsetstrokecolor{dialinecolor}
\draw (17.000000\du,13.000000\du)--(17.000000\du,20.000000\du);
}
\definecolor{dialinecolor}{rgb}{0.000000, 0.000000, 0.000000}
\pgfsetstrokecolor{dialinecolor}
\node[anchor=west] at (7.500000\du,20.500000\du){$t_0$};
\definecolor{dialinecolor}{rgb}{0.000000, 0.000000, 0.000000}
\pgfsetstrokecolor{dialinecolor}
\node[anchor=west] at (10.500000\du,20.500000\du){$t_1$};
\definecolor{dialinecolor}{rgb}{0.000000, 0.000000, 0.000000}
\pgfsetstrokecolor{dialinecolor}
\node[anchor=west] at (16.000000\du,20.500000\du){$t_{\max}$};
\definecolor{dialinecolor}{rgb}{0.000000, 0.000000, 0.000000}
\pgfsetstrokecolor{dialinecolor}
\node[anchor=west] at (5.200000\du,18.000000\du){$y$};
\end{tikzpicture} }
\caption{Claim (i)}
\end{subfigure}
\begin{subfigure}[b]{0.48 \textwidth}
\resizebox{\textwidth}{!}{  
\ifx\du\undefined
  \newlength{\du}
\fi
\setlength{\du}{15\unitlength}
\begin{tikzpicture}
\pgftransformxscale{1.000000}
\pgftransformyscale{-1.000000}
\definecolor{dialinecolor}{rgb}{0.000000, 0.000000, 0.000000}
\pgfsetstrokecolor{dialinecolor}
\definecolor{dialinecolor}{rgb}{1.000000, 1.000000, 1.000000}
\pgfsetfillcolor{dialinecolor}
\pgfsetlinewidth{0.100000\du}
\pgfsetdash{}{0pt}
\pgfsetdash{}{0pt}
\pgfsetbuttcap
{
\definecolor{dialinecolor}{rgb}{0.000000, 0.000000, 0.000000}
\pgfsetfillcolor{dialinecolor}
\definecolor{dialinecolor}{rgb}{0.000000, 0.000000, 0.000000}
\pgfsetstrokecolor{dialinecolor}
\draw (6.000000\du,12.000000\du)--(6.000000\du,27.000000\du);
}
\pgfsetlinewidth{0.100000\du}
\pgfsetdash{}{0pt}
\pgfsetdash{}{0pt}
\pgfsetbuttcap
{
\definecolor{dialinecolor}{rgb}{0.000000, 0.000000, 0.000000}
\pgfsetfillcolor{dialinecolor}
\definecolor{dialinecolor}{rgb}{0.000000, 0.000000, 0.000000}
\pgfsetstrokecolor{dialinecolor}
\draw (6.000000\du,20.000000\du)--(31.000000\du,20.000000\du);
}
\definecolor{dialinecolor}{rgb}{0.000000, 0.000000, 0.000000}
\pgfsetstrokecolor{dialinecolor}
\node[anchor=west] at (5.000000\du,16.000000\du){$y_1$};
\definecolor{dialinecolor}{rgb}{0.000000, 0.000000, 0.000000}
\pgfsetstrokecolor{dialinecolor}
\node[anchor=west] at (4.400000\du,24.000000\du){$-y_1$};
\definecolor{dialinecolor}{rgb}{0.000000, 0.000000, 0.000000}
\pgfsetstrokecolor{dialinecolor}
\node[anchor=west] at (3.000000\du,13.000000\du){$(1+\delta) y_1$};
\pgfsetlinewidth{0.000000\du}
\pgfsetdash{}{0pt}
\pgfsetdash{}{0pt}
\pgfsetmiterjoin
\definecolor{dialinecolor}{rgb}{0.678431, 0.847059, 0.901961}
\pgfsetfillcolor{dialinecolor}
\fill (6.000000\du,13.000000\du)--(6.000000\du,16.000000\du)--(30.000000\du,16.000000\du)--(30.000000\du,13.000000\du)--cycle;
\definecolor{dialinecolor}{rgb}{0.000000, 0.000000, 0.000000}
\pgfsetstrokecolor{dialinecolor}
\draw (6.000000\du,13.000000\du)--(6.000000\du,16.000000\du)--(30.000000\du,16.000000\du)--(30.000000\du,13.000000\du)--cycle;
\definecolor{dialinecolor}{rgb}{0.000000, 0.000000, 0.000000}
\pgfsetstrokecolor{dialinecolor}
\node[anchor=west] at (5.000000\du,18.000000\du){$y$};
\pgfsetlinewidth{0.100000\du}
\pgfsetdash{}{0pt}
\pgfsetdash{}{0pt}
\pgfsetbuttcap
{
\definecolor{dialinecolor}{rgb}{0.000000, 0.000000, 0.000000}
\pgfsetfillcolor{dialinecolor}
\definecolor{dialinecolor}{rgb}{0.000000, 0.000000, 0.000000}
\pgfsetstrokecolor{dialinecolor}
\draw (6.000000\du,16.000000\du)--(16.000000\du,26.000000\du);
}
\pgfsetlinewidth{0.100000\du}
\pgfsetdash{}{0pt}
\pgfsetdash{}{0pt}
\pgfsetbuttcap
{
\definecolor{dialinecolor}{rgb}{0.000000, 0.000000, 0.000000}
\pgfsetfillcolor{dialinecolor}
\definecolor{dialinecolor}{rgb}{0.000000, 0.000000, 0.000000}
\pgfsetstrokecolor{dialinecolor}
\draw (16.000000\du,26.000000\du)--(29.000000\du,13.000000\du);
}
\pgfsetlinewidth{0.000000\du}
\pgfsetdash{{\pgflinewidth}{0.200000\du}}{0cm}
\pgfsetdash{{\pgflinewidth}{0.200000\du}}{0cm}
\pgfsetbuttcap
{
\definecolor{dialinecolor}{rgb}{0.000000, 0.000000, 0.000000}
\pgfsetfillcolor{dialinecolor}
\definecolor{dialinecolor}{rgb}{0.000000, 0.000000, 0.000000}
\pgfsetstrokecolor{dialinecolor}
\draw (29.000000\du,13.000000\du)--(29.000000\du,20.000000\du);
}
\definecolor{dialinecolor}{rgb}{0.000000, 0.000000, 0.000000}
\pgfsetstrokecolor{dialinecolor}
\node[anchor=west] at (27.500000\du,21.000000\du){$t_{\max}$};
\pgfsetlinewidth{0.100000\du}
\pgfsetdash{}{0pt}
\pgfsetdash{}{0pt}
\pgfsetbuttcap
{
\definecolor{dialinecolor}{rgb}{0.000000, 0.000000, 0.000000}
\pgfsetfillcolor{dialinecolor}
\definecolor{dialinecolor}{rgb}{0.000000, 0.000000, 0.000000}
\pgfsetstrokecolor{dialinecolor}
\draw (6.000000\du,18.000000\du)--(13.000000\du,25.000000\du);
}
\pgfsetlinewidth{0.100000\du}
\pgfsetdash{}{0pt}
\pgfsetdash{}{0pt}
\pgfsetbuttcap
{
\definecolor{dialinecolor}{rgb}{0.000000, 0.000000, 0.000000}
\pgfsetfillcolor{dialinecolor}
\definecolor{dialinecolor}{rgb}{0.000000, 0.000000, 0.000000}
\pgfsetstrokecolor{dialinecolor}
\draw (13.000000\du,25.000000\du)--(25.000000\du,13.000000\du);
}
\pgfsetlinewidth{0.000000\du}
\pgfsetdash{{\pgflinewidth}{0.200000\du}}{0cm}
\pgfsetdash{{\pgflinewidth}{0.200000\du}}{0cm}
\pgfsetbuttcap
{
\definecolor{dialinecolor}{rgb}{0.000000, 0.000000, 0.000000}
\pgfsetfillcolor{dialinecolor}
\definecolor{dialinecolor}{rgb}{0.000000, 0.000000, 0.000000}
\pgfsetstrokecolor{dialinecolor}
\draw (12.000000\du,24.000000\du)--(6.000000\du,24.000000\du);
}
\pgfsetlinewidth{0.000000\du}
\pgfsetdash{{\pgflinewidth}{0.200000\du}}{0cm}
\pgfsetdash{{\pgflinewidth}{0.200000\du}}{0cm}
\pgfsetbuttcap
{
\definecolor{dialinecolor}{rgb}{0.000000, 0.000000, 0.000000}
\pgfsetfillcolor{dialinecolor}
\definecolor{dialinecolor}{rgb}{0.000000, 0.000000, 0.000000}
\pgfsetstrokecolor{dialinecolor}
\draw (12.000000\du,20.000000\du)--(12.000000\du,24.000000\du);
}
\definecolor{dialinecolor}{rgb}{0.000000, 0.000000, 0.000000}
\pgfsetstrokecolor{dialinecolor}
\node[anchor=west] at (11.000000\du,19.500000\du){$t_0$};
\pgfsetlinewidth{0.000000\du}
\pgfsetdash{{\pgflinewidth}{0.200000\du}}{0cm}
\pgfsetdash{{\pgflinewidth}{0.200000\du}}{0cm}
\pgfsetbuttcap
{
\definecolor{dialinecolor}{rgb}{0.000000, 0.000000, 0.000000}
\pgfsetfillcolor{dialinecolor}
\definecolor{dialinecolor}{rgb}{0.000000, 0.000000, 0.000000}
\pgfsetstrokecolor{dialinecolor}
\draw (13.000000\du,25.000000\du)--(13.000000\du,20.000000\du);
}
\definecolor{dialinecolor}{rgb}{0.000000, 0.000000, 0.000000}
\pgfsetstrokecolor{dialinecolor}
\node[anchor=west] at (13.000000\du,19.500000\du){$s$};
\definecolor{dialinecolor}{rgb}{0.000000, 0.000000, 0.000000}
\pgfsetstrokecolor{dialinecolor}
\node[anchor=west] at (14.500000\du,27.50000\du){$\tau$};
\pgfsetlinewidth{0.000000\du}
\pgfsetdash{{\pgflinewidth}{0.200000\du}}{0cm}
\pgfsetdash{{\pgflinewidth}{0.200000\du}}{0cm}
\pgfsetbuttcap
{
\definecolor{dialinecolor}{rgb}{0.000000, 0.000000, 0.000000}
\pgfsetfillcolor{dialinecolor}
\definecolor{dialinecolor}{rgb}{0.000000, 0.000000, 0.000000}
\pgfsetstrokecolor{dialinecolor}
\draw (16.000000\du,27.000000\du)--(16.000000\du,26.000000\du);
}
\pgfsetlinewidth{0.000000\du}
\pgfsetdash{{\pgflinewidth}{0.200000\du}}{0cm}
\pgfsetdash{{\pgflinewidth}{0.200000\du}}{0cm}
\pgfsetbuttcap
{
\definecolor{dialinecolor}{rgb}{0.000000, 0.000000, 0.000000}
\pgfsetfillcolor{dialinecolor}
\definecolor{dialinecolor}{rgb}{0.000000, 0.000000, 0.000000}
\pgfsetstrokecolor{dialinecolor}
\draw (14.000000\du,24.000000\du)--(14.000000\du,27.000000\du);
}
\end{tikzpicture}}
\caption{Claim (ii)}
\end{subfigure}
\caption{Illustration of the proof of Lemma~\ref{lem:compact-sets-are-petite}}
\end{figure}
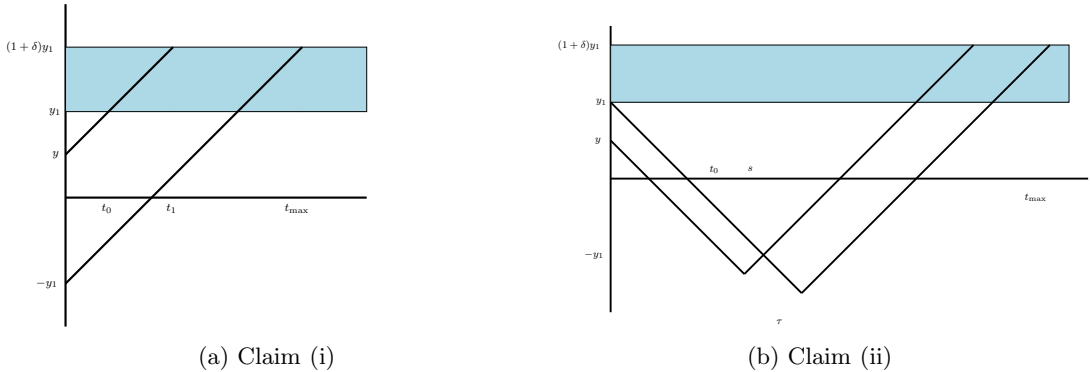

\emph{Claim (i): There exists a constant $c > 0$ such that for $\varphi \in B_b(E)$, $\varphi \geq 0$, and  $y \in [-y_1, y_1]$,~\eqref{eq:petite} holds for $j = +1$.}

\emph{Proof of Claim (i):} Let $j = +1$. Let $\lambda_{\max} := \max_{y \in [-y_1, (1+\delta)y_1]} \lambda(y,+1)$, which is finite by Assumption~\ref{ass:construction}(i). The time of reaching $(1+ \delta) y_1$ from $-y_1$ is $t_{\max} := (2+\delta) y_1$.
Let $c := \exp(-(\lambda_{\max} + 1) t_{\max})$.

Let $\varphi \in B_b(E)$, $\varphi \geq 0$, $y \in [-y_1,y_1]$ and $0 \leq t \leq t_{\max}$. Then
\begin{align*}
 \exp(-t)\E_{y,j} [ \varphi(Y(t),J(t))] & \geq \exp(-t)\E_{y,+1} \left[\varphi(Y(t),J(t)) \1_{\{T_1 \geq t\}} \right] \\
 & = \exp(-t) \varphi(y + t, +1) F(t;y,+1) \\
 & \geq \exp(-(\lambda_{\max} +1) t_{\max}) \varphi(y + t, + 1) = c \varphi(y + t, + 1).
\end{align*}
Define $t_0(y) := \inf\{ t \geq 0 : y + t = y_1 \} = y_1-y$ and $t_1(y) := \inf\{ t \geq 0: y + t = (1+ \delta) y_1\} = (1+\delta) y_1 - y \leq t_{\max}$. Then
\begin{align*} 
\int_0^{\infty} \exp(-t) \E_{y,j} [\varphi(Y(t),J(t))] \ d t & \geq c \int_{t_0}^{t_1} \varphi(y + t, +1) \ dt= c \int_{y_1}^{(1+\delta) y_1} \varphi(z,+1) \ d z.
\end{align*}
$\hfill \diamond$

\emph{Claim (ii): There exists a constant $c > 0$ such that for $\varphi \in B_b(E)$, $\varphi \geq 0$, and  $y \in [-y_1, y_1]$,~\eqref{eq:petite} holds for $j = -1$.}

\emph{Proof of Claim (ii):} Let $j = -1$. We know by Assumption~\ref{ass:construction}(ii), that $\lambda(z, -1)$ is bounded from below for $z \leq -y_1$ by $\lambda_{\min}$. Heuristically, in order to obtain a uniformly positive probability of switching to the $+$-replica, we need to spend at least a certain amount of time, $\tau > 0$ say, in the region $(-\infty, -y_1]$. For definiteness, let $ \tau \in (0, 2 y_1)$. Hence from a given $y \in [-y_1, y_1]$ we will travel for a certain amount of time $t_0(y) := y_1 + y$ until we reach $-y_1$, and then continue moving in the negative direction up to time $t_1(y) := t_0(y) + \tau$. We will then have to move back in the positive direction from $-y_1 - \tau$ until reaching $y_1(1+\delta)$. The maximum amount of time required is obtained if we start from $y = + y_1$, which results in a value $t_{\max} := (4 + \delta) y_1  + 2 \tau$.

Let 
\[ \lambda_{\max} := \sup_{0 \leq t \leq t_{\max}} \lambda(y-t,-1) \vee \lambda(y + t, +1) \vee \lambda(y-t,+1) \vee \lambda(y +t, +1),\]
a crude but effective upper bound for the switching rate in either replica. Define \[c := \lambda_{\min} \exp(-(\lambda_{\max} + 1) t_{\max}) \tau.\]
Let $\varphi \in B_b(E)$, $\varphi \geq 0$, $y \in [-y_1,y_1]$ and $0 \leq t \leq t_{\max}$. Then
\begin{equation}
\label{eq:estimate_one_switch} 
\begin{aligned} 
& \E_{y,j} [ \exp(-t) \varphi(Y(t),J(t)) \1_{\{T_1 \leq t, T_2 > t\}}]  \\
& = \E_{y,j} \left[\exp(-t) \int_0^t \varphi(y - s + (t-s),+1) \1_{\{T_1 \in d s , T_2 > t\}} \right]  \\
& = \exp(-t) \int_0^t \varphi(y - 2 s + t,+1) \lambda(y - s, -1) F(s;y,-1) F(t-s; y - s, +1) \ d s \\
& \geq \int_{t_0(y)}^t \varphi(y - 2 s + t,+1) \lambda_{\min} \exp(-(\lambda_{\max}+1)t) \ d s  =\frac{c }{\tau} \int_{t_0(y)}^t \varphi(y - 2 s + t,+1)  \ d s.
\end{aligned}
\end{equation}
We have
\begin{align*}
 & \int_0^{\infty} \E_{y,j} [ \exp(-t) \varphi(Y(t),J(t))] \ d t \\
 & \geq \int_0^{t_{\max}}  \E_{y,j} [ \exp(-t) \varphi(Y(t),J(t)) \1_{\{T_1 \leq t, T_2 > t\}}] \ dt \\
 & \geq c /\tau \int_0^{t_{\max}}  \int_{t_0(y)}^t \varphi(y-2 s + t, +1) \ d s \ d t  & \mbox{(by \eqref{eq:estimate_one_switch})}  \\
 & \geq c /\tau \int_{t_0(y)}^{t_1(y)} \int_{t_1(y)}^{t_{\max}}  \varphi(y-2s + t, +1) \ d s \ d t & \mbox{(Fubini, reduced integration area)} \\
& \geq c /\tau \int_{t_0(y)}^{t_1(y)} \int_{y_1-y + 2 s}^{(1+\delta)y_1 - y + 2 s}  \varphi(y-2s + t, +1) \ d t \ d s & \mbox{($\star$)} \\
 & = c /\tau \int_{t_0(y)}^{t_1(y)} \int_{y_1}^{(1+\delta) y_1} \varphi(z, +1) \ d z \ d s = c \int_{y_1}^{(1+\delta)y_1} \varphi(z, +1)  \ d z & \mbox{($z = y-2s +t$)}
\end{align*}
In the step labelled ($\star$) we have reduced the integration area: For $s \geq t_0(y)$, since $\tau \leq 2 y_1$,
\[ y_1 - y + 2s \geq y_1 - y+2 t_0(y) = 3 y_1 + y \geq y_1 + y + \tau = t_1(y),\]
and for $s \leq t_1(y)$,
\[ (1+\delta)y_1 - y + 2 s \leq (1 + \delta) y_1 - y + 2 (y + y_1) + 2 \tau  \leq (4 + \delta)y_1 + 2 \tau = t_{\max}. \]
This establishes the claim. $\hfill \diamond$

By taking the minimum over the constants $c$ obtained in Claims (i) and (ii), the inequality~\eqref{eq:petite} follows for all $(y,j) \in C$.
\end{proof}

\subsubsection{Foster-Lyapunov function}

The following lemma, in particular the choice of the Lyapunov function $V$, is based on the proof of \cite[Proposition 2.8]{Fontbona2015}.

\begin{lemma}[Existence of a Foster-Lyapunov function]
\label{lem:lyapunov}
Suppose Assumption~\ref{ass:lyapunov-improved} holds. Then there are constants $c > 0$ and $y_1>0$ and a continuously differentiable function $V: E \rightarrow (0,\infty)$ such that $V(y,j) \rightarrow \infty$ as $|y| \rightarrow \infty$, and $ L V(y,j) \leq - c V(y,j)$  for $(y,j) \in E$ with $|y| > y_1$.
\end{lemma}

\begin{proof}
Define 
\begin{align*}
m^+ & := \sup_{y \geq y_0} \lambda^-(y), \quad  & M^+ & := \inf_{y \geq y_0} \lambda^+(y), \\
m^- & := \sup_{y \leq - y_0} \lambda^+(y), \quad & M^- & := \inf_{y \leq - y_0} \lambda^-(y).
\end{align*}

By Assumption~\ref{ass:lyapunov-improved}, $M^+ > m^+$ and $M^- > m^-$. In particular there exist constants $\beta^+ > 0$, $\beta^- > 0$ such that $m^+ \exp(2 \beta^+) < M^+$ and $m^- \exp(2 \beta^-) < M^-$.
It follows that 
\[  m^{\pm}\left(\exp(2 \beta^{\pm}) - 1 \right) = m^{\pm} \exp(2 \beta^{\pm}) \left(1 - \exp(-2 \beta^{\pm}) \right)    <  M^{\pm} \left(1 - \exp(-2 \beta^{\pm}) \right).\]
Therefore we can pick positive constants $\alpha^{\pm} \in \left( m^{\pm} \left(\exp(2 \beta^{\pm}) - 1 \right), M^{\pm} \left(1 - \exp(-2 \beta^{\pm}) \right) \right)$. 
Let $y_1 \geq y_0$ be undefined for now. As a Lyapunov function we take a function $V$ such that, outside of $(-y_1,+y_1)$, and for $j \in \{-1,+1\}$,
\begin{equation} \label{eq:lyapunov_function} V(y,j) = \begin{cases} \exp( \alpha^+ y + \beta^+ \sign(j)), & \mbox{if} \ y \geq y_1, \\
             \exp(- \alpha^- y - \beta^- \sign(j)), & \mbox{if} \ y \leq - y_1.
            \end{cases}\end{equation}
and such that $V$ is positive and continuously differentiable on $(-y_1, +y_1)$. As long as $y_1 \geq y_0$ is taken sufficiently large then $V$ thus defined is positive and continuously differentiable on $E$.
Now on $y \geq y_1$, we have
\begin{align*} L V(y, +1) & = \left(\alpha^+ - \lambda^+(y) \left( 1-  \exp(- 2 \beta^+) \right) \right) V(y,+1), \\
 LV(y,-1) & = \left( - \alpha^+ + \lambda^-(y) \left( \exp(2 \beta^+) - 1 \right) \right)V(y,-1).
\end{align*}
By the choice of $\alpha^{\pm}$, we have
\begin{align*} \alpha^+ - \lambda^+(y) \left(1- \exp(- 2 \beta^+)\right) & \leq \alpha^+ - M^+ \left(1- \exp(- 2 \beta^+) \right)< 0 \quad \mbox{and} \\
- \alpha^+ + \lambda^-(y) \left( \exp(2 \beta^+) - 1 \right) & \leq -\alpha^+ + m^+ \left( \exp(2 \beta^+) - 1 \right) < 0. 
\end{align*}
It follows that there exists a constant $c^+ > 0$ such that $LV(y,j) \leq - c^+ V(y,j)$ for $y \geq y_1$ and $j \in \{-1,+1\}$. Analogously, there exists a constant $c^- > 0$ such that $LV(y,j) \leq - c^- V(y,j)$ for $y \leq - y_1$ and $j \in \{-1,+1\}$. The proof is completed by taking $c := c^- \wedge c^+$.
\end{proof}

\section*{Acknowledgements}
We acknowledge the suggestions for improvement of two anonymous referees and the associate editor, which have significantly contributed towards the accuracy and clarity of this work.

\end{document}